\documentclass[preprint,12pt]{elsarticle}

\usepackage[T1]{fontenc}
\usepackage[utf8]{inputenc}
\usepackage{amsmath,amssymb,amsthm,mathtools}
\usepackage{bm}
\usepackage{microtype}
\usepackage{graphicx}
\usepackage{xcolor}
\usepackage{enumitem}
\usepackage{booktabs,tabularx,array,float}
\usepackage{tikz}
\usetikzlibrary{arrows.meta,calc,positioning,decorations.pathreplacing}
\newcolumntype{Y}{>{\raggedright\arraybackslash}X}
\newcolumntype{L}[1]{>{\raggedright\arraybackslash}p{#1}}
\usepackage{hyperref}

\hypersetup{
  colorlinks=true,
  linkcolor=blue!50!black,
  citecolor=blue!50!black,
  urlcolor=blue!50!black,
  pdftitle={Scattering criteria for the three-dimensional focusing energy-critical generalized Hartree equation},
  pdfauthor={Pang-Hung Chung and Dan Han}
}

\journal{Nonlinear Analysis}
\biboptions{sort&compress}
\allowdisplaybreaks
\numberwithin{equation}{section}

\newtheorem{theorem}{Theorem}[section]
\newtheorem{proposition}[theorem]{Proposition}
\newtheorem{lemma}[theorem]{Lemma}
\newtheorem{corollary}[theorem]{Corollary}

\theoremstyle{remark}

\newcommand{\R}{\mathbb R}
\newcommand{\C}{\mathbb C}
\newcommand{\eps}{\varepsilon}
\newcommand{\dd}{\,d}
\newcommand{\Dcal}{\mathcal D_{\alpha}}
\newcommand{\Ecal}{E_{\alpha}}
\newcommand{\Kcal}{K_{\alpha}}
\newcommand{\Vcal}{\mathcal V}
\newcommand{\Ttail}[1]{\mathcal T_{\alpha,#1}}
\newcommand{\Ncal}{\mathcal N_{\alpha}}
\newcommand{\Znorm}{Z_{\alpha}}
\newcommand{\Wnorm}{\mathcal W_{\alpha}}
\newcommand{\Nnorm}{\mathfrak N_{\alpha}}
\newcommand{\one}{\mathbf 1}
\newcommand{\Err}{\operatorname{Err}}
\newcommand{\mopt}{\mathfrak m^{\mathrm{opt}}}
\newcommand{\Aopt}{A^{\mathrm{opt}}}
\newcommand{\Mcal}{\mathcal M}
\newcommand{\Pcal}{\mathcal P}

\begin{document}

\begin{frontmatter}

\title{Scattering Criteria for the Three-Dimensional Focusing Energy-Critical Generalized Hartree Equation}

\author[aff1]{Pang-Hung Chung}
\ead{penghongzhong@yahoo.com}

\author[aff2]{Akidul Haque}
\ead{akidul.haque@louisville.edu}

\author[aff2]{Dan Han\corref{cor1}}
\ead{dan.han@louisville.edu}
\cortext[cor1]{Corresponding author.}

\address[aff1]{Department of Applied Mathematics, Guangdong University of Education, Guangzhou 510640, P. R. China}
\address[aff2]{Department of Mathematics, University of Louisville, Louisville, KY 40245, USA}

\begin{abstract}
Two nonradial scattering criteria are established for the three-dimensional focusing energy-critical generalized Hartree equation below the ground-state threshold, under a bounded-scale concentration--compactness reduction. The first criterion is formulated in terms of fixed-center occupation windows and quantitative bounds on the motion of the concentration center. The second is based on uniform low-frequency \(L^2\)-decay, which yields finite mass, precompactness in the inhomogeneous energy space, a zero-momentum normalization, and sublinear center drift. The proofs combine a localized Hartree virial identity with a critical Hardy--Littlewood--Sobolev estimate controlling the nonlocal tail.
\end{abstract}


\end{frontmatter}

\section{Introduction}

Nonlinear Schr\"odinger equations with Hartree or Choquard interactions form a basic class of nonlocal dispersive models.  Their evolution is driven by a convolution potential generated by the solution itself, so the local dispersive mechanism must be combined with long-range spatial interaction.  Foundational well-posedness and scattering results, together with the Strichartz framework used in this paper, may be found in \cite{HayashiTsutsumi1987,GinibreVelo2000,Cazenave2003,KeelTao1998}.  On the variational side, the associated elliptic problem is governed by the sharp Hardy--Littlewood--Sobolev and Sobolev inequalities; optimizer theory and qualitative properties of Choquard ground states were developed in \cite{Lieb1983,Aubin1976,Talenti1976,MorozVanSchaftingen2013,LiLiuTangXu2026}.

At the energy-critical level, the classical quadratic-density Hartree equation has a comparatively mature scattering theory in dimensions five and higher.  Radial and nonradial defocusing scattering, focusing scattering below the ground-state threshold, and threshold dynamics have been established in a sequence of works \cite{MiaoXuZhao2007,MiaoXuZhao2011,MiaoXuZhao2009,LiMiaoZhang2009,MiaoWuXu2015}.  For generalized powers, local theory, sharp mass--energy dichotomies, radial scattering mechanisms, and threshold dynamics in intercritical regimes were developed in \cite{Arora2019,AroraRoudenko2020,AroraRoudenko2022,Zhou2025}.  In three dimensions, however, the homogeneous focusing generalized energy-critical problem retains a substantial nonradial difficulty: radial symmetry fixes the concentration center, whereas translation invariance allows a nonradial compact solution to drift through space.  The corresponding inhomogeneous model contains a spatial weight that penalizes such escape, but that pinning mechanism is absent in the homogeneous equation \cite{GuzmanXu2023}.

The purpose of this paper is to isolate two mechanisms that rule out bounded-scale compact obstructions to scattering in this translation-invariant setting.  One mechanism is geometric and measures how often the concentration center remains in a fixed spatial window; the other is spectral and converts uniform low-frequency decay into finite mass, zero momentum, and sublinear center motion.  Both mechanisms ultimately feed into a localized virial argument, but they require different information about the compact critical element.

We now introduce the equation studied here.  Fix $0<\alpha<3$, set $p:=3+\alpha$, and let
\begin{equation*}
  I_\alpha(x):=c_\alpha|x|^{-(3-\alpha)},
  \qquad c_\alpha>0,
\end{equation*}
be the Riesz kernel on $\R^3$.  For an interval $I\subset\R$ containing the origin and initial data $u_0\in\dot H^1(\R^3)$, we consider
\begin{equation}\label{eq:Hartree}
  \left\{
  \begin{aligned}
  i\partial_tu+\Delta u+\bigl(I_\alpha*|u|^p\bigr)|u|^{p-2}u&=0,
  &&(t,x)\in I\times\R^3,\\
  u(0,x)&=u_0(x),
  &&x\in\R^3.
  \end{aligned}
  \right.
\end{equation}
The relation $p=3+\alpha$ is precisely the energy-critical relation.  Indeed, for $\lambda>0$, the rescaling
\[
  u_\lambda(t,x):=\lambda^{1/2}u(\lambda^2t,\lambda x)
\]
preserves both \eqref{eq:Hartree} and the homogeneous energy norm:
\[
  \|u_\lambda(t)\|_{\dot H^1(\R^3)}
  =\|u(\lambda^2t)\|_{\dot H^1(\R^3)}.
\]

We use standard notation for Lebesgue and Sobolev spaces, mixed space--time norms, Fourier multipliers, and the free Schr\"odinger group $e^{it\Delta}$.  For a measurable set $E$, its indicator and Lebesgue measure are denoted by $\one_E$ and $|E|$.  We write
\[
  X\lesssim_{\Lambda}Y
\]
when $X\le C_{\Lambda}Y$ for a positive constant depending only on the parameters in $\Lambda$; an unsubscripted implicit constant may change from line to line.  The notation $F(T)=o(G(T))$ means $F(T)/G(T)\to0$ as $T\to\infty$, and $\mathcal C\Subset X$ means that $\mathcal C$ is relatively compact in the Banach space $X$.

For $f\in\dot H^1(\R^3)$ define
\begin{align*}
  \Dcal(f)
  &:=\int_{\R^3}(I_\alpha*|f|^p)(x)|f(x)|^p\dd x\\
  &=\iint_{\R^3\times\R^3}
  I_\alpha(x-y)|f(x)|^p|f(y)|^p\dd x\dd y.
\end{align*}
\begin{equation}\label{eq:energy}
  \Ecal(f):=\frac12\|\nabla f\|_2^2-\frac1{2p}\Dcal(f),
\end{equation}
and
\begin{equation}\label{eq:Kalpha}
  \Kcal(f):=\|\nabla f\|_2^2-\Dcal(f).
\end{equation}
The bilinear Hardy--Littlewood--Sobolev inequality and Sobolev embedding give
\begin{align}
  \Dcal(f)
  &\lesssim_\alpha
  \bigl\||f|^p\bigr\|_{L^{6/p}}^2
  =\|f\|_{L^6}^{2p}
  \lesssim_\alpha\|\nabla f\|_{L^2}^{2p}.
  \label{eq:D-finite}
\end{align}
Define the optimal Hartree--Sobolev constant by
\begin{equation*}
  C_{\rm HS}(\alpha)
  :=\sup_{0\ne f\in\dot H^1(\R^3)}
  \frac{\Dcal(f)}{\|\nabla f\|_2^{2p}}.
\end{equation*}
Let $W_\alpha\in\dot H^1(\R^3)$ be a positive optimizer attaining $C_{\rm HS}(\alpha)$, normalized so that it solves
\begin{equation}\label{eq:ground-state}
  -\Delta W_\alpha
  =(I_\alpha*|W_\alpha|^p)|W_\alpha|^{p-2}W_\alpha.
\end{equation}
Testing \eqref{eq:ground-state} against $W_\alpha$ and using the Pohozaev identity give
\begin{equation}\label{eq:ground-pohozaev}
  \|\nabla W_\alpha\|_2^2=\Dcal(W_\alpha),
  \qquad
  \Ecal(W_\alpha)=\frac{p-1}{2p}\|\nabla W_\alpha\|_2^2.
\end{equation}

Let $d$ denote the spatial dimension. The energy-critical Hartree theory with quadratic density $p=2$ begins when $d\ge5$.  Radial defocusing scattering, nonradial defocusing scattering, radial focusing threshold scattering, nonradial focusing scattering below the kinetic threshold, and the threshold dynamics were established in a sequence of works \cite{MiaoXuZhao2007,MiaoXuZhao2011,MiaoXuZhao2009,LiMiaoZhang2009,MiaoWuXu2015}.  For general powers, the local theory, sharp mass--energy dichotomies, radial scattering mechanisms, and threshold dynamics in the intercritical range were developed in \cite{Arora2019,AroraRoudenko2020,AroraRoudenko2022,Zhou2025}.

The generalized equation considered here differs from the classical theory in two essential respects: it lies in the three-dimensional energy-critical regime $p=3+\alpha>3$, and its homogeneous nonlinearity preserves translation invariance.  In the three-dimensional generalized energy-critical problem, radial threshold scattering is available, whereas the homogeneous nonradial problem is left open by the known theory \cite{GuzmanXu2023}.  The spatial weight in the inhomogeneous model penalizes an escaping concentration center; no analogous pinning is available for \eqref{eq:Hartree}.

For comparison with the homogeneous equation, let $b>0$ and write
\begin{equation*}
  p_b:=3+\alpha-2b.
\end{equation*}
The corresponding three-dimensional inhomogeneous generalized Hartree nonlinearity is
\[
  |x|^{-b}\bigl(I_\alpha*(|\cdot|^{-b}|u|^{p_b})\bigr)|u|^{p_b-2}u.
\]
In the two tables that follow, ``classical Hartree'' refers to the quadratic-density energy-critical model $p=2$ in spatial dimensions at least five, whereas ``generalized Hartree'' refers to the energy-critical exponent $p=1+(2+\alpha)/(d-2)$, which reduces to $p=3+\alpha$ when $d=3$.

\begin{table}[H]
\centering
\caption{Selected scattering milestones for the classical energy-critical Hartree equation.}
\label{tab:classical-scattering-progress}
\small
\setlength{\tabcolsep}{4pt}
\renewcommand{\arraystretch}{1.10}
\begin{tabularx}{\textwidth}{@{}L{0.24\textwidth}L{0.18\textwidth}Y@{}}
\toprule
Regime & Data & Main scattering result \\
\midrule
Defocusing \cite{MiaoXuZhao2007,MiaoXuZhao2011}
& Radial; general
& Global finite-energy scattering, with radial symmetry subsequently removed. \\
Focusing below threshold \cite{MiaoXuZhao2009,LiMiaoZhang2009}
& Radial; general
& Scattering on the subthreshold kinetic branch and the complementary blow-up alternative. \\
Threshold energy \cite{MiaoWuXu2015}
& Radial
& Classification of threshold dynamics under the stated spectral hypothesis. \\
\bottomrule
\end{tabularx}
\end{table}

\begin{table}[H]
\centering
\caption{Scattering results closest to the three-dimensional generalized problem.}
\label{tab:generalized-scattering-progress}
\small
\setlength{\tabcolsep}{4pt}
\renewcommand{\arraystretch}{1.10}
\begin{tabularx}{\textwidth}{@{}L{0.26\textwidth}L{0.18\textwidth}Y@{}}
\toprule
Model & Data & Main scattering result \\
\midrule
Defocusing generalized Hartree \cite{Saanouni2021}
& Radial
& Energy-critical global well-posedness and scattering via Morawetz control. \\
Three-dimensional inhomogeneous model, $b>0$ \cite{GuzmanXu2023}
& General
& Focusing scattering below the ground-state threshold; the homogeneous focusing result there is radial. \\
Present work, $d=3$ and $p=3+\alpha$
& Nonradial bounded-scale compact branch
& Fixed-center occupation-window and low-frequency scattering criteria. \\
\bottomrule
\end{tabularx}
\end{table}

Tables~\ref{tab:classical-scattering-progress} and~\ref{tab:generalized-scattering-progress} separate the quadratic-density theory, where nonradial threshold scattering is known, from the genuinely generalized three-dimensional focusing problem addressed here.  They also show why translation-invariant center motion, rather than the nonlocal convolution alone, is the principal obstruction in the homogeneous nonradial setting.

Our analysis is formulated within the concentration-compactness/rigidity framework.  Its profile-decomposition and minimal-element foundations originate in \cite{BahouriGerard1999,Keraani2001,KenigMerle2006}; the nonradial energy-critical Schr\"odinger developments in \cite{CollianderEtAl2008,KillipVisan2010,Dodson2019} show both the strength of the method and the significance of controlling spatial drift.

Our setting differs from the preceding Hartree scattering theory in three structural respects.  First, radial arguments fix the concentration center, while the inhomogeneous factor $|x|^{-b}$ in \cite{GuzmanXu2023} penalizes translation away from the origin.  Neither mechanism is available for the homogeneous nonradial equation \eqref{eq:Hartree}.  Second, in contrast with the quadratic-density energy-critical model treated in \cite{MiaoXuZhao2009,LiMiaoZhang2009,MiaoXuZhao2011}, the generalized nonlinearity with $p=3+\alpha>3$ produces genuinely two-point localization errors in the virial identity.  We control these errors at the critical regularity by combining the bilinear Hardy--Littlewood--Sobolev inequality with the uniform $L^6$ tail supplied by compactness.  Third, we do not impose differentiability, bounded speed, or a finite first moment on the modulation center $x(t)$.

The two main contributions are complementary.  The first is a fixed-center occupation-window method: instead of requiring pointwise control of $x(t)$ at every time, it allows a small exceptional set of times and measures the smallest fixed ball that captures the center for most of a long interval.  A normalized local-mass density modulus then improves the virial endpoint bound from $O(R^2)$ to $o(R^2)$ along diffusive radii.  The second is an infrared route: uniform low-frequency $L^2$ decay upgrades homogeneous compactness to $H^1$ precompactness, yields finite mass, permits a zero-momentum Galilean normalization, and recovers sublinear center drift from truncated barycenters.  Thus the two criteria isolate, respectively, a geometric occupation mechanism and a spectral low-frequency mechanism for excluding bounded-scale compact obstructions to scattering.

Throughout the paper the threshold assumptions are
\begin{equation}\label{eq:below-threshold}
  \Ecal(u_0)<\Ecal(W_\alpha),
  \qquad
  \|\nabla u_0\|_2<\|\nabla W_\alpha\|_2.
\end{equation}
They imply a uniform potential-well gap for every nonzero solution considered below.

For an interval $J\subset\R$, define the critical spacetime norm
\begin{equation}\label{eq:Z-norm-intro}
  \|u\|_{\Znorm(J)}
  :=\|u\|_{L_t^{2(p-1)}L_x^{\frac{6(p-1)}{p-3}}(J\times\R^3)}.
\end{equation}
A strong solution of \eqref{eq:Hartree} on $J$ is a function satisfying
\[
  u\in C(J;\dot H^1(\R^3)),
  \qquad
  \|u\|_{\Znorm(J)}<\infty,
\]
and, for every $t_0,t\in J$,
\begin{equation}\label{eq:Duhamel-intro}
  u(t)=e^{i(t-t_0)\Delta}u(t_0)
  +i\int_{t_0}^{t}e^{i(t-s)\Delta}
  \bigl[(I_\alpha*|u(s)|^p)|u(s)|^{p-2}u(s)\bigr]\dd s
\end{equation}
in $\dot H^1(\R^3)$.  Its maximal lifespan is denoted by $I_{\max}$.
We say that a maximal-lifespan solution satisfies the \emph{two-sided scattering conclusion} if
\begin{equation}\label{eq:global-and-finite-Z-intro}
  I_{\max}=\R,
  \qquad
  \|u\|_{\Znorm(\R)}<\infty,
\end{equation}
and there exist $u_\pm\in\dot H^1(\R^3)$ such that
\begin{equation}\label{eq:two-sided-scattering-intro}
  \lim_{t\to\pm\infty}
  \|u(t)-e^{it\Delta}u_\pm\|_{\dot H^1}=0.
\end{equation}
Equivalently, the scattering states in \eqref{eq:two-sided-scattering-intro} satisfy
\begin{equation}\label{eq:duhamel-scattering-intro}
  u(t)=e^{it\Delta}u_\pm
  -i\int_t^{\pm\infty}e^{i(t-s)\Delta}
  \bigl[(I_\alpha*|u(s)|^p)|u(s)|^{p-2}u(s)\bigr]\dd s
\end{equation}
in $\dot H^1(\R^3)$.

Let $U$ be a forward-global strong solution and write $U_0:=U(0)$. We call $U$ a \emph{bounded-scale compact critical element} if $U_0$ satisfies \eqref{eq:below-threshold} and there exist measurable functions $N:[0,\infty)\to(0,\infty)$ and $x:[0,\infty)\to\R^3$, together with constants $0<N_-\le N_+<\infty$, such that
\begin{equation}\label{eq:AP-intro}
  \left\{
  N(t)^{-1/2}U\left(t,x(t)+\frac{\cdot}{N(t)}\right):t\ge0
  \right\}
  \Subset\dot H^1(\R^3)
\end{equation}
and
\begin{equation}\label{eq:bounded-scale-intro}
  0<N_-\le N(t)\le N_+<\infty,
  \qquad t\ge0.
\end{equation}
Consequently,
\begin{equation}\label{eq:translated-orbit-intro}
  \mathcal K_x
  :=\{U(t,\cdot+x(t)):t\ge0\}
  \Subset\dot H^1(\R^3).
\end{equation}
No continuity or differentiability of $x(t)$ is assumed.

For $A\ge0$ and $T>0$ define
\begin{equation}\label{eq:mopt-def}
  \mopt_T(A)
  :=\inf_{Z\in\R^3}\frac1T
  \left|\{t\in[0,T]:|x(t)-Z|>A\}\right|,
\end{equation}
\begin{equation}\label{eq:Aopt-def}
  \Aopt_T(\delta)
  :=\inf\{A\ge0:\mopt_T(A)\le\delta\},
  \qquad 0<\delta<1,
\end{equation}
and
\begin{equation}\label{eq:z-Dx}
  z:=x(0),
  \qquad
  D_x(T):=\sup_{0\le t\le T}|x(t)-z|.
\end{equation}
For $2\le q\le6$, set
\begin{equation}\label{eq:theta}
  \theta(q):=3\left(\frac12-\frac1q\right),
  \qquad
  M_q(U):=\sup_{t\ge0}\|U(t)\|_{L^q}.
\end{equation}
The density modulus used at the virial endpoints is
\begin{equation}\label{eq:density-modulus-intro}
  \omega_U(R)
  :=R^{-2}\sup_{t\ge0}
  \int_{|y-x(t)|\le R}|U(t,y)|^2\dd y,
  \qquad R\ge1.
\end{equation}

For $\nu>0$, define the sharp low-frequency Fourier projection $P_{\le\nu}$ by
\[
  \widehat{P_{\le\nu}f}(\xi)
  :=\one_{\{|\xi|\le\nu\}}\widehat f(\xi).
\]
For $s>0$, define $|\nabla|^{-s}$ by
\[
  \widehat{|\nabla|^{-s}f}(\xi):=|\xi|^{-s}\widehat f(\xi)
\]
whenever the right-hand side defines an $L^2$ function. Whenever $f\in H^1(\R^3)$, define its mass and momentum components by
\begin{equation}\label{eq:mass-momentum-intro}
  \Mcal(f):=\int_{\R^3}|f(x)|^2\dd x,
  \qquad
  \Pcal_j(f):=\operatorname{Im}
  \int_{\R^3}\overline{f(x)}\,\partial_jf(x)\dd x,
  \quad j\in\{1,2,3\},
\end{equation}
and set
\begin{equation*}
  \Pcal(f):=\bigl(\Pcal_1(f),\Pcal_2(f),\Pcal_3(f)\bigr)\in\R^3.
\end{equation*}
For every $H^1$ solution $u$ of \eqref{eq:Hartree}, $\Mcal(u(t))$ and $\Pcal(u(t))$ are independent of $t$. We write $\Mcal(u)$ and $\Pcal(u)$ for these conserved values.

Let $\mathfrak A\subset\dot H^1(\R^3)$ be a class of initial data satisfying \eqref{eq:below-threshold}.  We say that $\mathfrak A$ admits a \emph{bounded-scale compactness reduction} if, whenever the maximal-lifespan solution associated with some $u_0\in\mathfrak A$ fails to satisfy \eqref{eq:global-and-finite-Z-intro}--\eqref{eq:two-sided-scattering-intro}, time reversal when necessary produces a nonzero forward-global bounded-scale compact critical element.

We now state the two scattering criteria. The first theorem gives a geometric scattering criterion: scattering follows when every bounded-scale compact obstruction is confined, for most times, to a fixed ball of at most diffusive radius, or when its maximal drift is quantitatively smaller than the virial scale dictated by an available $L^q$ bound.

\begin{theorem}[Fixed-center occupation-window and drift scattering criterion]\label{thm:window-scattering}
Let $\mathfrak A\subset\dot H^1(\R^3)$ admit a bounded-scale compactness reduction. Assume that every bounded-scale compact critical element $U$ furnished by this reduction satisfies at least one of the following alternatives.
\begin{enumerate}[label=\textup{(\roman*)},leftmargin=2.4em]
\item There exist sequences $T_n\to\infty$ and $\delta_n\downarrow0$ such that
\begin{equation}\label{eq:window-scattering-condition}
  \sup_{n\ge1}
  \frac{\Aopt_{T_n}(\delta_n)}{T_n^{1/2}}<\infty.
\end{equation}
\item For some $q\in[2,6]$, one has $M_q(U)<\infty$ and
\begin{equation}\label{eq:drift-scattering-condition}
  \liminf_{T\to\infty}
  \frac{\bigl(D_x(T)+1\bigr)^{1+\theta(q)}}{T}=0.
\end{equation}
\end{enumerate}
Then every $u_0\in\mathfrak A$ generates a maximal-lifespan solution satisfying the two-sided scattering conclusion \eqref{eq:global-and-finite-Z-intro}--\eqref{eq:two-sided-scattering-intro}.  In alternative \textup{(ii)}, the following explicit drift rates are sufficient:
\begin{equation}\label{eq:drift-rate-examples}
  \begin{aligned}
  q=6&:\quad D_x(T)=o(T^{1/2}),\\
  q=4&:\quad D_x(T)=o(T^{4/7}),\\
  q=2&:\quad D_x(T)=o(T).
  \end{aligned}
\end{equation}
\end{theorem}

The preceding criterion uses direct spatial information on the modulation center.  The next theorem provides a complementary spectral entrance and requires no a priori drift estimate: decay of the very low frequencies creates the finite-mass compactness needed to derive the center motion from conservation laws.

\begin{theorem}[Low-frequency scattering criterion]\label{thm:lowfreq-scattering}
Let $\mathfrak A\subset\dot H^1(\R^3)$ admit a bounded-scale compactness reduction. Assume that every bounded-scale compact critical element $U$ furnished by this reduction satisfies
\begin{equation}\label{eq:lowfreq-main}
  \lim_{\nu\downarrow0}
  \sup_{t\ge0}\|P_{\le\nu}U(t)\|_2=0.
\end{equation}
Then every $u_0\in\mathfrak A$ generates a maximal-lifespan solution satisfying \eqref{eq:global-and-finite-Z-intro}--\eqref{eq:two-sided-scattering-intro}.  In particular, the same conclusion holds if every critical element furnished by the reduction obeys, for some $s=s(U)>0$,
\begin{equation}\label{eq:negative-regularity-entrance}
  \sup_{t\ge0}\||\nabla|^{-s}U(t)\|_2<\infty.
\end{equation}
\end{theorem}

The two criteria can be placed under one virial-admissibility constraint.  Fix a compactness radius $L>0$ and, for $T,R>0$ and $Z\in\R^3$, define the exceptional-time set
\begin{equation}\label{eq:unified-bad-set}
  \mathcal B_U(T;Z,R,L)
  :=\bigl\{t\in[0,T]:|x(t)-Z|>R-L\bigr\}.
\end{equation}
The two endpoint estimates used in the paper are represented by
\begin{equation}\label{eq:unified-endpoint-costs}
  \mathfrak E_U^{\rm comp}(R)
  :=R^2\omega_U(3R)^{1/2},
  \qquad
  \mathfrak E_{U,q}^{L^q}(R)
  :=M_q(U)R^{1+\theta(q)},
  \quad 2\le q\le6.
\end{equation}
Thus both theorems reduce the exclusion of a bounded-scale compact critical element to finding sequences
\[
  T_n\to\infty,
  \qquad Z_n\in\R^3,
  \qquad R_n\ge2L,
\]
for which one of the endpoint costs in \eqref{eq:unified-endpoint-costs}, denoted by $\mathfrak E_U(R_n)$, satisfies the single asymptotic constraint
\begin{equation}\label{eq:unified-virial-constraint}
  \frac{|\mathcal B_U(T_n;Z_n,R_n,L)|}{T_n}
  +\frac{\mathfrak E_U(R_n)}{T_n}
  \longrightarrow0.
\end{equation}
Indeed, outside the set in \eqref{eq:unified-bad-set}, the fixed truncation centered at $Z_n$ contains the compact core, while the first term in \eqref{eq:unified-virial-constraint} absorbs the uniform loss on exceptional times and the second term absorbs the two virial endpoints.  The resulting positive bulk contribution is comparable to $T_n$, contradicting \eqref{eq:unified-virial-constraint}.

For Theorem~\ref{thm:window-scattering}\textup{(i)}, choose $Z_n$ and $R_n$ from the optimal occupation window, so that
\begin{align}\label{eq:window-to-unified}
  \frac{|\mathcal B_U(T_n;Z_n,R_n,L)|}{T_n}
  &\le\delta_n\longrightarrow0,\\
  \frac{\mathfrak E_U^{\rm comp}(R_n)}{T_n}
  &=\frac{R_n^2}{T_n}\,\omega_U(3R_n)^{1/2}
  \longrightarrow0,\notag
\end{align}
Thus \eqref{eq:window-to-unified} verifies the common constraint because $R_n=O(T_n^{1/2})$ and $\omega_U(R)\to0$.  For Theorem~\ref{thm:window-scattering}\textup{(ii)}, take
\begin{equation}\label{eq:drift-to-unified}
  Z_n:=x(0),
  \qquad
  R_n:=D_x(T_n)+L,
\end{equation}
so that $\mathcal B_U(T_n;Z_n,R_n,L)=\varnothing$ and
\[
  \frac{\mathfrak E_{U,q}^{L^q}(R_n)}{T_n}
  \lesssim
  M_q(U)\frac{(D_x(T_n)+1)^{1+\theta(q)}}{T_n}
  \longrightarrow0.
\]

Theorem~\ref{thm:lowfreq-scattering} reaches the same constraint without assuming center geometry.  Its hypothesis first yields
\begin{align}\label{eq:lowfreq-to-drift-summary}
  U&\in L_t^\infty H_x^1,
  &\{U(t,\cdot+x(t)):t\ge0\}&\Subset H^1,\notag\\
  \Pcal(U)&=0,
  &D_x(T)&=o(T).
\end{align}
after a Galilean normalization.  The implications summarized in \eqref{eq:lowfreq-to-drift-summary} are proved in Section~7.  Consequently, setting $q=2$ in \eqref{eq:unified-endpoint-costs} and using the choice in \eqref{eq:drift-to-unified} gives
\begin{equation}\label{eq:lowfreq-to-unified}
  \frac{\mathfrak E_{U,2}^{L^2}(R_n)}{T_n}
  =\Mcal(U)\frac{R_n}{T_n}
  \longrightarrow0,
  \qquad
  \mathcal B_U(T_n;Z_n,R_n,L)=\varnothing.
\end{equation}
Equation~\eqref{eq:lowfreq-to-unified} therefore shows that the low-frequency theorem supplies, rather than assumes, the $q=2$ realization of the common relation \eqref{eq:unified-virial-constraint}.  The distinction between the two results lies only in how that relation is obtained: Theorem~\ref{thm:window-scattering} enters through spatial occupation or quantitative drift, whereas Theorem~\ref{thm:lowfreq-scattering} enters through infrared decay and derives the required finite-mass drift structure.

The remainder is organized as follows. Section~2 develops local estimates and coercivity; Section~3 treats compactness tails and center geometry; Section~4 gives local-mass bounds; Section~5 derives the localized Hartree virial identity; Sections~6 and~7 prove the two exclusion mechanisms; and Section~8 completes the scattering deductions.

\section{Critical local estimates and variational coercivity}

All spatial norms in this section are taken over $\R^3$. A pair $(q,r)$ is called three-dimensional Schr\"odinger admissible if
\begin{equation}\label{eq:admissible}
  2\le q,r\le\infty,
  \qquad (q,r)\ne(2,\infty),
  \qquad \frac2q+\frac3r=\frac32.
\end{equation}
The endpoint Strichartz estimates used below are due to Keel and Tao \cite{KeelTao1998}. With admissibility defined by \eqref{eq:admissible}, for an interval $J\subset\R$ set
\begin{equation}\label{eq:S1-norm}
  \|u\|_{\mathcal S^1(J)}
  :=\sup_{(q,r)\text{ admissible}}
  \|\nabla u\|_{L_t^qL_x^r(J\times\R^3)}.
\end{equation}
For the norm in \eqref{eq:S1-norm}, the endpoint inhomogeneous Strichartz estimate gives, for $t_0\in J$ and every forcing term $F:J\times\R^3\to\C$ with finite right-hand side,
\begin{equation}\label{eq:inhom-Strichartz}
  \left\|\int_{t_0}^t e^{i(t-s)\Delta}F(s)\dd s\right\|_{\mathcal S^1(J)}
  \lesssim
  \|\nabla F\|_{L_t^2L_x^{6/5}(J\times\R^3)}.
\end{equation}
We abbreviate
\begin{equation}\label{eq:N-norm}
  \|F\|_{\Nnorm(J)}
  :=\|\nabla F\|_{L_t^2L_x^{6/5}(J\times\R^3)}.
\end{equation}

For later use, introduce the exponents
\begin{equation}\label{eq:critical-exponents}
  Q:=2(p-1),
  \qquad R_*:=\frac{6(p-1)}{p-3},
  \qquad r_*:=\frac{6(p-1)}{3p-5}.
\end{equation}
Define
\begin{equation}\label{eq:W-norm}
  \|u\|_{\Wnorm(J)}
  :=\|\nabla u\|_{L_t^Q L_x^{r_*}(J\times\R^3)}.
\end{equation}
Then $(Q,r_*)$ is admissible and
\begin{equation}\label{eq:critical-exponent-relations}
  \frac2Q+\frac3{r_*}=\frac32,
  \qquad
  \frac1{R_*}=\frac1{r_*}-\frac13,
  \qquad
  \frac2Q+\frac3{R_*}=\frac12.
\end{equation}
The exponents in \eqref{eq:critical-exponents} satisfy \eqref{eq:critical-exponent-relations}. Consequently, using \eqref{eq:Z-norm-intro} and \eqref{eq:W-norm},
\begin{equation}\label{eq:W-below-S}
  \|u\|_{\Wnorm(J)}\le \|u\|_{\mathcal S^1(J)},
  \qquad
  \|u\|_{\Znorm(J)}\lesssim \|u\|_{\Wnorm(J)}.
\end{equation}

The forcing norm used below is defined in \eqref{eq:N-norm}. Define the nonlinear mapping
\[
  \Ncal(u):=(I_\alpha*|u|^p)|u|^{p-2}u.
\]
A strong solution on an interval $J\subset\R$ is a function satisfying
\[
  u\in C(J;\dot H^1(\R^3)),
  \qquad
  \|u\|_{\Znorm(J)}<\infty,
\]
and, for every $t_0,t\in J$,
\[
  u(t)=e^{i(t-t_0)\Delta}u(t_0)
  +i\int_{t_0}^{t}e^{i(t-s)\Delta}\Ncal(u(s))\dd s
\]
in $\dot H^1(\R^3)$. The local theory may be constructed from the next estimate; see also the general critical-regularity theory in \cite{AroraRoudenko2020}. Only the finite-norm scattering consequence is required below.

\begin{lemma}[Critical generalized Hartree estimate]\label{lem:critical-Hartree-estimate}
Let $J\subset\R$, and let $u:J\times\R^3\to\C$ be a measurable function for which the right-hand side below is finite. For $0<\alpha<3$ and $p=3+\alpha$,
\begin{equation}\label{eq:critical-Hartree-estimate}
  \|\Ncal(u)\|_{\Nnorm(J)}
  \lesssim_\alpha
  \|u\|_{L_t^\infty L_x^6(J)}^p
  \|u\|_{\Znorm(J)}^{p-2}
  \|u\|_{\Wnorm(J)}.
\end{equation}
In particular, by Sobolev embedding and \eqref{eq:W-below-S},
\begin{equation}\label{eq:critical-Hartree-estimate-S}
  \|\Ncal(u)\|_{\Nnorm(J)}
  \lesssim_\alpha
  \|u\|_{L_t^\infty\dot H_x^1(J)}^p
  \|u\|_{\Znorm(J)}^{p-2}
  \|u\|_{\mathcal S^1(J)}.
\end{equation}
\end{lemma}

\begin{proof}
The pointwise chain rule gives
\begin{align}
  |\nabla\Ncal(u)|
  \lesssim_p{}&
  (I_\alpha*|u|^p)|u|^{p-2}|\nabla u|\notag\\
  &+\bigl(I_\alpha*(|u|^{p-1}|\nabla u|)\bigr)|u|^{p-1}.
  \label{eq:Hartree-gradient-split}
\end{align}
Since $p=3+\alpha$, the fractional integration identity
\begin{equation}\label{eq:HLS-mapping-relation}
  \frac{6-p}{6}=\frac p6-\frac\alpha3
\end{equation}
implies, by the Hardy--Littlewood--Sobolev inequality, that for every $f\in L_x^{6/p}(\R^3)$,
\begin{equation}\label{eq:HLS-map-critical}
  \|I_\alpha*f\|_{L_x^{6/(6-p)}}
  \lesssim_\alpha \|f\|_{L_x^{6/p}}.
\end{equation}
Here \eqref{eq:HLS-mapping-relation} gives precisely the exponent relation in \eqref{eq:HLS-map-critical}; the assumptions $3<p<6$ place both exponents in the Hardy--Littlewood--Sobolev range.

For the first term in \eqref{eq:Hartree-gradient-split}, note that
\begin{equation}\label{eq:outside-exponent-identity}
  \frac{p-2}{R_*}+\frac1{r_*}=\frac{p-1}{6}
\end{equation}
and use \eqref{eq:outside-exponent-identity} together with H\"older's inequality in time to obtain
\begin{align}
 &\|(I_\alpha*|u|^p)|u|^{p-2}\nabla u\|_{L_t^2L_x^{6/5}(J)}\notag\\
 &\quad\le
 \|I_\alpha*|u|^p\|_{L_t^\infty L_x^{6/(6-p)}(J)}
 \||u|^{p-2}\nabla u\|_{L_t^2L_x^{6/(p-1)}(J)}\notag\\
 &\quad\lesssim_\alpha
 \|u\|_{L_t^\infty L_x^6(J)}^p
 \|u\|_{L_t^Q L_x^{R_*}(J)}^{p-2}
 \|\nabla u\|_{L_t^Q L_x^{r_*}(J)}.
 \label{eq:first-gradient-term}
\end{align}
Indeed, $(p-2)/Q+1/Q=1/2$.

For the second term, one has
\begin{equation}\label{eq:inside-exponent-identity}
  \frac16+\frac{p-2}{R_*}+\frac1{r_*}=\frac p6.
\end{equation}
Using \eqref{eq:inside-exponent-identity},
\begin{align}
 &\||u|^{p-1}\nabla u\|_{L_t^2L_x^{6/p}(J)}\notag\\
 &\quad\le
 \|u\|_{L_t^\infty L_x^6(J)}
 \|u\|_{L_t^Q L_x^{R_*}(J)}^{p-2}
 \|\nabla u\|_{L_t^Q L_x^{r_*}(J)}.
 \label{eq:inside-convolution-bound}
\end{align}
Applying \eqref{eq:HLS-map-critical} to \eqref{eq:inside-convolution-bound} and using
\begin{equation*}
  \||u|^{p-1}\|_{L_t^\infty L_x^{6/(p-1)}}
  =\|u\|_{L_t^\infty L_x^6}^{p-1}
\end{equation*}
gives
\begin{align}
 &\|\bigl(I_\alpha*(|u|^{p-1}|\nabla u|)\bigr)|u|^{p-1}\|_{L_t^2L_x^{6/5}(J)}\notag\\
 &\quad\lesssim_\alpha
 \|u\|_{L_t^\infty L_x^6(J)}^p
 \|u\|_{\Znorm(J)}^{p-2}
 \|u\|_{\Wnorm(J)}.
 \label{eq:second-gradient-term}
\end{align}
Combining \eqref{eq:first-gradient-term} and \eqref{eq:second-gradient-term} proves \eqref{eq:critical-Hartree-estimate}; \eqref{eq:critical-Hartree-estimate-S} follows from \eqref{eq:W-below-S} and Sobolev embedding.
\end{proof}

The preceding nonlinear estimate yields the standard finite-critical-norm scattering implication.

\begin{proposition}[Finite critical norm implies scattering]\label{prop:finite-Z-scattering}
Let $u$ be a global strong solution of \eqref{eq:Hartree} on $[0,\infty)$ such that
\begin{equation}\label{eq:finite-Z-hypotheses}
  H:=\sup_{t\ge0}\|u(t)\|_{\dot H^1}<\infty,
  \qquad
  \|u\|_{\Znorm([0,\infty))}<\infty.
\end{equation}
Then there exists $u_+\in\dot H^1(\R^3)$ satisfying
\begin{equation}\label{eq:forward-scattering}
  \lim_{t\to\infty}
  \|u(t)-e^{it\Delta}u_+\|_{\dot H^1}=0.
\end{equation}
The analogous assertion holds backward in time.
\end{proposition}

\begin{proof}
Let $\eta>0$ be a parameter to be fixed below in terms of $\alpha$ and $H$. Under \eqref{eq:finite-Z-hypotheses}, absolute continuity of the $L_t^Q L_x^{R_*}$ norm yields an integer $K\ge1$ and consecutive intervals $J_1,\dots,J_K$ whose union is $[0,\infty)$ and such that
\begin{equation}\label{eq:Z-small-partition}
  \|u\|_{\Znorm(J_k)}\le\eta,\qquad 1\le k\le K.
\end{equation}
For each $k\in\{1,\dots,K\}$ choose $t_k\in J_k$. Using the partition property \eqref{eq:Z-small-partition}, Duhamel's formula, \eqref{eq:inhom-Strichartz}, and Lemma~\ref{lem:critical-Hartree-estimate}, fix a constant $C_*=C_*(\alpha)\ge1$ such that
\begin{align}
  \|u\|_{\mathcal S^1(J_k)}
  &\le C_*\|u(t_k)\|_{\dot H^1}
  +C_*\|\Ncal(u)\|_{\Nnorm(J_k)}\notag\\
  &\le C_*H+C_*H^p\eta^{p-2}
  \|u\|_{\mathcal S^1(J_k)}.
  \label{eq:S-absorption}
\end{align}
Choose $\eta$ so that $C_*H^p\eta^{p-2}\le1/2$ in \eqref{eq:S-absorption}. Then
\begin{equation}\label{eq:S-on-pieces}
  \|u\|_{\mathcal S^1(J_k)}\lesssim_{\alpha,H}1.
\end{equation}
Summing the bounds \eqref{eq:S-on-pieces} over the finite partition yields
\begin{equation}\label{eq:global-S-finite}
  \|u\|_{\mathcal S^1([0,\infty))}<\infty.
\end{equation}

For $0<T_1<T_2$, Duhamel's formula and \eqref{eq:inhom-Strichartz} imply
\begin{align}
 &\|e^{-iT_2\Delta}u(T_2)-e^{-iT_1\Delta}u(T_1)\|_{\dot H^1}\notag\\
 &\quad\lesssim
 \|\Ncal(u)\|_{\Nnorm([T_1,T_2])}\notag\\
 &\quad\lesssim_\alpha
 H^p\|u\|_{\Znorm([T_1,T_2])}^{p-2}
 \|u\|_{\mathcal S^1([T_1,T_2])}.
 \label{eq:scattering-Cauchy}
\end{align}
The right-hand side of \eqref{eq:scattering-Cauchy} tends to zero as $T_1,T_2\to\infty$, because the global $\Znorm$ norm is finite and \eqref{eq:global-S-finite} holds. Thus $e^{-it\Delta}u(t)$ is Cauchy in $\dot H^1$; denote its limit by $u_+$. Reversing Duhamel's formula proves \eqref{eq:forward-scattering}.
\end{proof}

We next record the sharp variational information. Put
\begin{equation}\label{eq:A-AW}
  A(f):=\|\nabla f\|_{L^2}^2,
  \qquad A_\alpha:=A(W_\alpha).
\end{equation}
With the notation in \eqref{eq:A-AW}, the sharp Hartree--Sobolev inequality normalized by $W_\alpha$ is
\begin{equation}\label{eq:sharp-HLS-Sobolev}
  \Dcal(f)\le A_\alpha^{1-p}A(f)^p.
\end{equation}
Indeed, equality holds at $W_\alpha$ by \eqref{eq:ground-pohozaev}. If
\begin{equation}\label{eq:y-def}
  y:=\frac{A(f)}{A_\alpha},
\end{equation}
define, for $y\ge0$,
\begin{equation*}
  g_p(y):=\frac12y-\frac1{2p}y^p.
\end{equation*}
Substituting \eqref{eq:y-def} into \eqref{eq:energy}, \eqref{eq:Kalpha}, and \eqref{eq:sharp-HLS-Sobolev} then gives
\begin{align}
  \Ecal(f)&\ge A_\alpha g_p(y),
  \label{eq:energy-lower-gp}\\
  \Kcal(f)&\ge A_\alpha(y-y^p).
  \label{eq:K-lower-y-Hartree}
\end{align}
Moreover,
\begin{equation}\label{eq:E-A-K}
  \Ecal(f)=\frac{p-1}{2p}A(f)+\frac1{2p}\Kcal(f).
\end{equation}

The variational inequalities above first isolate the subthreshold kinetic branch.

\begin{lemma}[Below-threshold branch]\label{lem:below-threshold-branch}
Let $u$ be a strong solution of \eqref{eq:Hartree} on an interval containing $0$, and assume \eqref{eq:below-threshold}. Then there exists $\delta_1=\delta_1(u_0,\alpha)>0$ such that throughout the lifespan,
\begin{equation}\label{eq:uniform-gradient-gap}
  A(u(t))\le(1-\delta_1)A_\alpha.
\end{equation}
Set
\begin{equation*}
  c_1:=1-(1-\delta_1)^{p-1}>0.
\end{equation*}
Then
\begin{equation}\label{eq:K-gradient-coercive}
  \Kcal(u(t))\ge c_1A(u(t)).
\end{equation}
In particular, the $\dot H^1$ norm stays uniformly bounded away from the ground-state threshold.
\end{lemma}

\begin{proof}
Set
\begin{equation*}
  y(t):=\frac{A(u(t))}{A_\alpha},
  \qquad
  e_0:=\frac{\Ecal(u_0)}{A_\alpha}.
\end{equation*}
By \eqref{eq:energy-lower-gp}, the initial inequalities imply
\begin{align*}
  0
  &< g_p(y(0))
  \le \frac{\Ecal(u_0)}{A_\alpha}
  =e_0
  <\frac{\Ecal(W_\alpha)}{A_\alpha}
  =g_p(1),
\end{align*}
where the strict positivity follows from $u_0\not\equiv0$; the zero solution is immediate.  Since
\begin{equation}\label{eq:proof-gp-derivative}
  g_p'(y)=\frac12(1-y^{p-1})>0,
  \qquad 0\le y<1,
\end{equation}
there is a unique $y_*\in(0,1)$ such that
\begin{equation}\label{eq:proof-ystar}
  g_p(y_*)=e_0,
  \qquad
  y(0)\le y_*<1,
  \qquad
  \delta_1:=1-y_*>0.
\end{equation}
Energy conservation and \eqref{eq:energy-lower-gp} yield, for every time in the lifespan,
\begin{equation}\label{eq:proof-gp-upper}
  g_p(y(t))
  \le \frac{\Ecal(u(t))}{A_\alpha}
  =\frac{\Ecal(u_0)}{A_\alpha}
  =e_0.
\end{equation}
We first verify that the orbit cannot cross the ground-state kinetic level.  If it did, continuity of $t\mapsto A(u(t))$ and $y(0)<1$ would provide a first time $t_1$ with $y(t_1)=1$.  Then
\begin{align*}
  g_p(1)
  &=g_p(y(t_1))
  \le \frac{\Ecal(u(t_1))}{A_\alpha}
  =e_0
  <g_p(1),
\end{align*}
which is impossible.  Hence $0\le y(t)<1$ throughout the lifespan.  On this interval $g_p$ is strictly increasing by \eqref{eq:proof-gp-derivative}; combining this fact with \eqref{eq:proof-gp-upper} and \eqref{eq:proof-ystar} gives
\begin{equation}\label{eq:proof-y-gap}
  y(t)\le y_*=1-\delta_1.
\end{equation}
This proves \eqref{eq:uniform-gradient-gap}.

Finally, using \eqref{eq:K-lower-y-Hartree} and \eqref{eq:proof-y-gap},
\begin{align*}
  \Kcal(u(t))
  &\ge A_\alpha\bigl(y(t)-y(t)^p\bigr)\notag\\
  &=A_\alpha y(t)\bigl(1-y(t)^{p-1}\bigr)\notag\\
  &\ge A_\alpha y(t)
  \bigl[1-(1-\delta_1)^{p-1}\bigr]\notag\\
  &=\bigl[1-(1-\delta_1)^{p-1}\bigr]A(u(t))
  =c_1A(u(t)),
\end{align*}
which is \eqref{eq:K-gradient-coercive}.
\end{proof}

The branch coercivity can be strengthened to a strictly positive time-uniform virial gap for every nonzero solution.

\begin{lemma}[Positive virial gap]\label{lem:positive-virial-gap}
Under the assumptions of Lemma~\ref{lem:below-threshold-branch}, if $u\not\equiv0$, then
\begin{equation}\label{eq:kappa-def}
  \kappa_0:=2c_1\Ecal(u_0)>0
\end{equation}
and
\begin{equation}\label{eq:K-positive-gap}
  \Kcal(u(t))\ge\kappa_0
\end{equation}
throughout the lifespan.
\end{lemma}

\begin{proof}
For $0<y<1$, $g_p(y)>0$. Equivalently, \eqref{eq:E-A-K} and \eqref{eq:K-gradient-coercive} imply $\Ecal(u_0)>0$ for a nontrivial solution. Since $\Dcal(u(t))\ge0$, energy conservation gives
\begin{equation}\label{eq:A-lower-energy}
  \Ecal(u_0)=\Ecal(u(t))\le\frac12A(u(t)),
  \qquad A(u(t))\ge2\Ecal(u_0).
\end{equation}
Combining \eqref{eq:K-gradient-coercive} with \eqref{eq:A-lower-energy} proves \eqref{eq:K-positive-gap}.
\end{proof}

\section{Critical compactness, nonlocal tails, and drift geometry}

Assume throughout this section that $U$, $N$, and $x$ satisfy \eqref{eq:AP-intro} and \eqref{eq:bounded-scale-intro}. Define the normalized orbit
\begin{equation}\label{eq:normalized-orbit}
  v(t,y):=N(t)^{-1/2}U\left(t,x(t)+\frac{y}{N(t)}\right)
\end{equation}
and its closure
\begin{equation}\label{eq:compact-K}
  \mathcal K:=\overline{\{v(t):t\ge0\}}^{\dot H^1}.
\end{equation}
By \eqref{eq:compact-K} and the almost-periodicity assumption, $\mathcal K$ is compact in $\dot H^1(\R^3)$.

We begin the compactness analysis by removing the bounded scale parameter from the normalized orbit.

\begin{lemma}[Translated-orbit compactness]\label{lem:translated-orbit-compact}
Under \eqref{eq:AP-intro} and \eqref{eq:bounded-scale-intro}, the translated orbit $\mathcal K_x$ defined in \eqref{eq:translated-orbit-intro} is precompact in $\dot H^1(\R^3)$.
\end{lemma}

\begin{proof}
For $\lambda>0$ define $\mathcal S_\lambda f(x):=\lambda^{1/2}f(\lambda x)$. From \eqref{eq:normalized-orbit},
\[
  U(t,x+x(t))=\mathcal S_{N(t)}v(t,x).
\]
The map $(\lambda,f)\mapsto\mathcal S_\lambda f$ is continuous from $(0,\infty)\times\dot H^1$ to $\dot H^1$. Hence the image of the compact set $[N_-,N_+]\times\mathcal K$ is compact and contains $\mathcal K_x$.
\end{proof}

The next elementary compactness statement converts precompactness into uniform spatial and frequency tails.

\begin{lemma}[Uniform tails of a compact set]\label{lem:compact-set-tails}
If $\mathcal C\Subset\dot H^1(\R^3)$, then
\begin{equation}\label{eq:compact-set-tail-limit}
  \lim_{A\to\infty}
  \sup_{f\in\mathcal C}
  \int_{|x|\ge A}\bigl(|\nabla f(x)|^2+|f(x)|^6\bigr)\dd x=0.
\end{equation}
\end{lemma}

\begin{proof}
Fix $\varepsilon>0$. Since $\mathcal C$ is compact, there exist an integer $J=J(\varepsilon)\ge1$ and functions
\begin{equation*}
  f_1,\dots,f_J\in\mathcal C
\end{equation*}
such that for every $f\in\mathcal C$ one can choose $j=j(f)\in\{1,\dots,J\}$ with
\begin{equation*}
  \|\nabla(f-f_j)\|_2
  \le\varepsilon,
  \qquad
  \|f-f_j\|_6\lesssim\varepsilon.
\end{equation*}
For this finite family choose $A_\varepsilon$ so large that
\begin{equation*}
  \max_{1\le j\le J}
  \int_{|x|\ge A_\varepsilon}
  \bigl(|\nabla f_j|^2+|f_j|^6\bigr)\dd x
  \le\varepsilon^2+\varepsilon^6.
\end{equation*}
Then, for $A\ge A_\varepsilon$ and $f\in\mathcal C$,
\begin{align*}
 &\int_{|x|\ge A}
  \bigl(|\nabla f|^2+|f|^6\bigr)\dd x\notag\\
 &\quad\le
  2\|\nabla(f-f_j)\|_2^2
  +2\int_{|x|\ge A}|\nabla f_j|^2\dd x\notag\\
 &\qquad
  +2^5\|f-f_j\|_6^6
  +2^5\int_{|x|\ge A}|f_j|^6\dd x\notag\\
 &\quad\le C\bigl(\varepsilon^2+\varepsilon^6\bigr).
\end{align*}
Taking the supremum over $f\in\mathcal C$ and then sending $\varepsilon\downarrow0$ proves \eqref{eq:compact-set-tail-limit}.
\end{proof}

Applying the preceding compact-set statement to the translated orbit gives the physical tails used throughout the virial argument.

\begin{lemma}[Physical kinetic and critical tails]\label{lem:physical-tails}
For every $\eps>0$ there exists $A_{\rm c}=A_{\rm c}(\eps)\ge1$ such that
\begin{equation}\label{eq:physical-tail-small}
  \sup_{t\ge0}
  \int_{|x-x(t)|\ge A_{\rm c}}
  \bigl(|\nabla U(t,x)|^2+|U(t,x)|^6\bigr)\dd x
  \le\eps.
\end{equation}
Equivalently, the tail modulus
\begin{equation}\label{eq:omega-def}
  \omega_{\rm c}(A)
  :=\sup_{t\ge0}
  \int_{|x-x(t)|\ge A}
  \bigl(|\nabla U(t,x)|^2+|U(t,x)|^6\bigr)\dd x
\end{equation}
satisfies $\omega_{\rm c}(A)\to0$ as $A\to\infty$.
\end{lemma}

\begin{proof}
Let $\varepsilon>0$.  By Lemma~\ref{lem:compact-set-tails}, there is $L=L(\varepsilon)\ge1$ such that
\begin{equation}\label{eq:normalized-tail-choice}
  \sup_{t\ge0}
  \int_{|y|\ge L}
  \bigl(|\nabla v(t,y)|^2+|v(t,y)|^6\bigr)\dd y
  \le\varepsilon.
\end{equation}
Set
\begin{equation*}
  A_{\rm c}:=\max\left\{1,\frac{L}{N_-}\right\}.
\end{equation*}
For $x=x(t)+y/N(t)$, one has
\begin{equation*}
  |x-x(t)|\ge A_{\rm c}
  \quad\Longrightarrow\quad
  |y|=N(t)|x-x(t)|\ge N_-A_{\rm c}\ge L.
\end{equation*}
The critical scaling gives, term by term,
\begin{align}
 &\int_{|x-x(t)|\ge A_{\rm c}}|\nabla U(t,x)|^2\dd x\notag\\
 &\qquad=
  \int_{|y|\ge N(t)A_{\rm c}}|\nabla v(t,y)|^2\dd y
  \le\int_{|y|\ge L}|\nabla v(t,y)|^2\dd y,
  \label{eq:physical-gradient-change}\\
 &\int_{|x-x(t)|\ge A_{\rm c}}|U(t,x)|^6\dd x\notag\\
 &\qquad=
  \int_{|y|\ge N(t)A_{\rm c}}|v(t,y)|^6\dd y
  \le\int_{|y|\ge L}|v(t,y)|^6\dd y.
  \label{eq:physical-L6-change}
\end{align}
Adding \eqref{eq:physical-gradient-change} and \eqref{eq:physical-L6-change}, taking the supremum in $t$, and invoking \eqref{eq:normalized-tail-choice} proves \eqref{eq:physical-tail-small}.  Since $\varepsilon$ is arbitrary, \eqref{eq:omega-def} tends to zero.
\end{proof}

The next estimate is the nonlocal substitute for the elementary fact that a local power nonlinearity has no interaction between two disjoint regions. For $f\in\dot H^1(\R^3)$, $z\in\R^3$, and $R>0$, define
\begin{align}\label{eq:Hartree-tail-definition}
  \Ttail{R}[f;z]
  :=\iint_{\substack{|x-z|\ge R\ \text{or}\ |y-z|\ge R}}
  I_\alpha(x-y)|f(x)|^p|f(y)|^p\dd x\dd y.
\end{align}

The genuinely nonlocal core--far-field interaction is controlled by the same critical $L^6$ tail, as follows.

\begin{lemma}[Critical HLS control of the Hartree tail]\label{lem:Hartree-tail-HLS}
For the tail functional defined in \eqref{eq:Hartree-tail-definition} and every $f\in\dot H^1(\R^3)$, $z\in\R^3$, and $R>0$,
\begin{equation}\label{eq:Hartree-tail-HLS}
  \Ttail{R}[f;z]
  \lesssim_\alpha
  \|f\|_{L^6}^{p}
  \|f\one_{\{|x-z|\ge R\}}\|_{L^6}^{p}.
\end{equation}
Consequently, under the compactness assumptions,
\begin{equation}\label{eq:moving-Hartree-tail-small}
  \lim_{A\to\infty}
  \sup_{t\ge0}\Ttail{A}[U(t);x(t)]=0.
\end{equation}
\end{lemma}

\begin{proof}
Let $\Omega_R:=\{x:|x-z|\ge R\}$. By symmetry of the kernel,
\begin{align*}
  \Ttail{R}[f;z]
  &\le2\iint I_\alpha(x-y)
  \one_{\Omega_R}(x)|f(x)|^p|f(y)|^p\dd x\dd y.
\end{align*}
The bilinear Hardy--Littlewood--Sobolev inequality applies with equal exponent $6/p$, because
\begin{equation*}
  \frac p6+\frac p6+\frac{3-\alpha}{3}=2
\end{equation*}
when $p=3+\alpha$. Hence
\begin{align*}
  \Ttail{R}[f;z]
  &\lesssim_\alpha
  \|\one_{\Omega_R}|f|^p\|_{L^{6/p}}
  \||f|^p\|_{L^{6/p}}\\
  &=\|f\one_{\Omega_R}\|_{L^6}^p\|f\|_{L^6}^p,
\end{align*}
which is \eqref{eq:Hartree-tail-HLS}. The global $L^6$ norm of $U(t)$ is uniformly bounded by Lemma~\ref{lem:below-threshold-branch} and Sobolev embedding, while the moving $L^6$ tail tends to zero by Lemma~\ref{lem:physical-tails}. This proves \eqref{eq:moving-Hartree-tail-small}.
\end{proof}

The fourth-order derivative of the virial cutoff produces a local $L^2$ annular term. It is controlled solely by the critical $L^6$ tail.

\begin{lemma}[Annular local-mass estimate]\label{lem:annular-L2}
For $f\in L^6(\R^3)$, $z\in\R^3$, and $R>0$,
\begin{equation}\label{eq:annular-L2}
  R^{-2}\int_{R\le|x-z|\le2R}|f(x)|^2\dd x
  \lesssim
  \left(\int_{|x-z|\ge R}|f(x)|^6\dd x\right)^{1/3}.
\end{equation}
\end{lemma}

\begin{proof}
Write
\[
  \Omega_{R,z}:=\{x\in\R^3:R\le|x-z|\le2R\}.
\]
Since
\begin{equation*}
  |\Omega_{R,z}|=\frac{4\pi}{3}(2^3-1)R^3=\frac{28\pi}{3}R^3,
\end{equation*}
H\"older's inequality with exponents $3/2$ and $3$ yields
\begin{align}
  \int_{\Omega_{R,z}}|f|^2\dd x
  &=\int_{\Omega_{R,z}}|f|^2\cdot1\dd x\notag\\
  &\le
  \left(\int_{\Omega_{R,z}}|f|^6\dd x\right)^{1/3}
  |\Omega_{R,z}|^{2/3}\notag\\
  &\le C R^2
  \left(\int_{|x-z|\ge R}|f(x)|^6\dd x\right)^{1/3}.
  \label{eq:annular-L2-expanded}
\end{align}
Dividing \eqref{eq:annular-L2-expanded} by $R^2$ gives \eqref{eq:annular-L2}.
\end{proof}

We next pass from a moving compactness center to a fixed center. Recall $z=x(0)$ and $D_x(T)$ from \eqref{eq:z-Dx}. Given a core radius $A\ge1$, define
\begin{equation}\label{eq:RT-def}
  R_T:=D_x(T)+A,
  \qquad T\ge0.
\end{equation}
For every $0\le t\le T$, the triangle inequality yields
\begin{align}
  B(x(t),A)&\subset B(z,R_T),
  \notag\\
  \{|x-z|\ge R_T\}&\subset\{|x-x(t)|\ge A\},
  \label{eq:fixed-moving-tail-inclusion}\\
  B(z,2R_T)&\subset B(x(t),3R_T).
  \notag
\end{align}
Fix $T>0$, a center $Z\in\R^3$, an occupation radius $A\ge0$, and a tolerance $0<\delta<1$. Denote the exceptional set by
\[
  E_{T,Z,A}:=\{t\in[0,T]:|x(t)-Z|>A\}.
\]
For the occupation-window schematic, assume $|E_{T,Z,A}|\le\delta T$.
To label the schematic compact core below, fix a tolerance $0<\eps_{\rm c}<1$ and choose $L_{\rm c}\ge1$ so that
\begin{equation*}
  \sup_{t\ge0}
  \int_{|x-x(t)|\ge L_{\rm c}}
  \bigl(|\nabla U(t,x)|^2+|U(t,x)|^6\bigr)\dd x
  \le\eps_{\rm c},
\end{equation*}
as permitted by Lemma~\ref{lem:physical-tails}. Finally, choose a fixed virial radius $R\ge A+L_{\rm c}$.
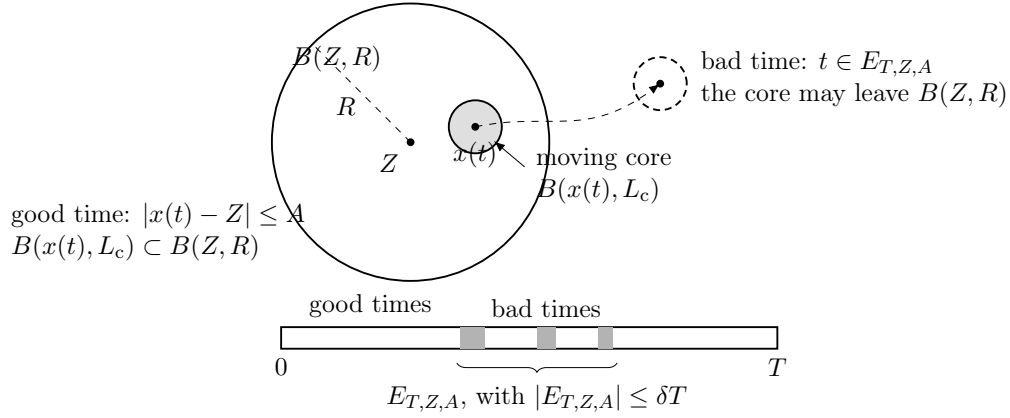
\begin{figure}[H]
\centering
\resizebox{0.98\textwidth}{!}{%
\begin{tikzpicture}[
  scale=0.92,
  every node/.style={font=\small},
  >=Latex
]
  \coordinate (Z) at (-1,0.7);
  \draw[thick] (Z) circle (2.25);
  \fill (Z) circle (1.8pt);
  \node[below left=1pt and 1pt of Z] {$Z$};
  \draw[dashed] (Z) -- ++(135:2.25);
  \node at (-2.2,2.05) {$B(Z,R)$};
  \node at (-2.05,1.28) {$R$};

  \coordinate (Xg) at (0.05,0.95);
  \draw[fill=gray!25,thick] (Xg) circle (0.43);
  \fill (Xg) circle (1.8pt);
  \node[below=2pt of Xg] {$x(t)$};
  \draw[->] (0.85,0.3) -- (0.35,0.72);
  \node[align=left,anchor=west] at (0.88,0.15) {moving core\\$B(x(t),L_{\rm c})$};
  \node[align=left,anchor=east] at (-2.55,-0.75) {good time: $|x(t)-Z|\le A$\\$B(x(t),L_{\rm c})\subset B(Z,R)$};

  \coordinate (Xb) at (3.05,1.65);
  \draw[densely dashed,thick] (Xb) circle (0.43);
  \fill (Xb) circle (1.8pt);
  \draw[dashed,->] (Xg) .. controls (1.25,1.2) and (1.85,0.75) .. (Xb);
  \node[align=left,anchor=west] at (3.55,1.72) {bad time: $t\in E_{T,Z,A}$\\the core may leave $B(Z,R)$};

  \draw[thick] (-3.1,-2.65) rectangle (4.95,-2.30);
  \fill[gray!60] (-0.20,-2.65) rectangle (0.20,-2.30);
  \fill[gray!60] (1.05,-2.65) rectangle (1.35,-2.30);
  \fill[gray!60] (2.05,-2.65) rectangle (2.28,-2.30);
  \node[below] at (-3.1,-2.65) {$0$};
  \node[below] at (4.95,-2.65) {$T$};
  \node[above] at (-1.65,-2.30) {good times};
  \node[above,align=center] at (1.2,-2.30) {bad times};
  \draw[decorate,decoration={brace,mirror,amplitude=4pt}]
    (-0.25,-2.88) -- (2.35,-2.88)
    node[midway,below=5pt,align=center] {$E_{T,Z,A}$, with $|E_{T,Z,A}|\le\delta T$};
\end{tikzpicture}%
}
\caption{Fixed-center occupation geometry. At good times the moving compact core $B(x(t),L_{\rm c})$ is contained in the fixed virial ball $B(Z,R)$ once $R$ dominates $A+L_{\rm c}$. The dark portions of the time strip represent the exceptional set $E_{T,Z,A}$, whose relative measure is small.}
\label{fig:fixed-center-geometry}
\end{figure}
With the geometry in Figure~\ref{fig:fixed-center-geometry} fixed, the moving-center tails transfer to the fixed center in the following form.

\begin{lemma}[Fixed-center tail control]\label{lem:fixed-center-tails}
For every $\eps>0$ there exists $A=A(\eps)\ge1$ such that, for every $T\ge0$, every $0\le t\le T$, and $R_T$ defined by \eqref{eq:RT-def},
\begin{align}
  &\int_{|x-z|\ge R_T}
  \bigl(|\nabla U(t,x)|^2+|U(t,x)|^6\bigr)\dd x
  \le\eps,
  \label{eq:fixed-tail-small}\\
  &R_T^{-2}\int_{R_T\le|x-z|\le2R_T}|U(t,x)|^2\dd x
  \le\eps,
  \label{eq:fixed-annulus-small}\\
  &\Ttail{R_T}[U(t);z]\le\eps.
  \label{eq:fixed-Hartree-tail-small}
\end{align}
\end{lemma}

\begin{proof}
Let
\begin{equation*}
  H_6:=\sup_{t\ge0}\|U(t)\|_6<\infty.
\end{equation*}
Choose $0<\eta<1$ so that
\begin{equation*}
  \eta\le\varepsilon,
  \qquad
  C\eta^{1/3}\le\varepsilon,
  \qquad
  C_\alpha H_6^p\eta^{p/6}\le\varepsilon.
\end{equation*}
By Lemma~\ref{lem:physical-tails}, choose $A=A(\eta)$ such that
\begin{equation}\label{eq:moving-tail-eta}
  \sup_{t\ge0}
  \int_{|x-x(t)|\ge A}
  \bigl(|\nabla U(t,x)|^2+|U(t,x)|^6\bigr)\dd x
  \le\eta.
\end{equation}
For $0\le t\le T$, \eqref{eq:fixed-moving-tail-inclusion} and \eqref{eq:moving-tail-eta} give
\begin{align}
  \int_{|x-z|\ge R_T}
  \bigl(|\nabla U(t,x)|^2+|U(t,x)|^6\bigr)\dd x
  &\le\eta\le\varepsilon.
  \label{eq:fixed-tail-first-chain}
\end{align}
Next, Lemma~\ref{lem:annular-L2} yields
\begin{align*}
  R_T^{-2}
  \int_{R_T\le|x-z|\le2R_T}|U(t,x)|^2\dd x
  &\le C
  \left(\int_{|x-z|\ge R_T}|U(t,x)|^6\dd x\right)^{1/3}\notag\\
  &\le C\eta^{1/3}\le\varepsilon.
\end{align*}
Finally, Lemma~\ref{lem:Hartree-tail-HLS} gives
\begin{align}
  \Ttail{R_T}[U(t);z]
  &\le C_\alpha\|U(t)\|_6^p
  \|U(t)\one_{\{|x-z|\ge R_T\}}\|_6^p\notag\\
  &\le C_\alpha H_6^p\eta^{p/6}
  \le\varepsilon.
  \label{eq:fixed-tail-third-chain}
\end{align}
Equations \eqref{eq:fixed-tail-first-chain}--\eqref{eq:fixed-tail-third-chain} prove the three conclusions simultaneously.
\end{proof}

\section{Local mass and preliminary endpoint estimates}

Throughout this section, $U$ denotes a bounded-scale compact critical element. Although solutions are only assumed to lie in the homogeneous energy space, their mass on a fixed bounded set is finite. For $f\in L^q(\R^3)$, $2\le q\le6$, $y\in\R^3$, and $R>0$, H\"older's inequality gives
\begin{equation}\label{eq:local-L2-Lq}
  \|f\|_{L^2(B(y,R))}
  \le |B(y,R)|^{\frac12-\frac1q}\|f\|_{L^q}
  \lesssim R^{\theta(q)}\|f\|_{L^q},
\end{equation}
where $\theta(q)$ is defined in \eqref{eq:theta}. In particular, if $f\in\dot H^1(\R^3)$, Sobolev embedding yields the purely energetic estimate
\begin{equation}\label{eq:local-L2-energy}
  \|f\|_{L^2(B(y,R))}
  \lesssim R\|\nabla f\|_{L^2}.
\end{equation}

The gauge-invariant Hartree nonlinearity contributes no source to the local mass continuity equation. This gives a useful consistency check on the local $L^2$ quantity.

\begin{lemma}[Fixed-center local mass variation]\label{lem:local-mass-variation}
Let $\chi\in C_c^\infty(\R^3)$ be real valued, let $R>0$ and $y\in\R^3$, and set
\begin{equation*}
  \mathfrak m_{R,y}(t)
  :=\int_{\R^3}\chi\left(\frac{x-y}{R}\right)|U(t,x)|^2\dd x.
\end{equation*}
Then
\begin{equation}\label{eq:localized-mass-derivative}
  \frac{\dd}{\dd t}\mathfrak m_{R,y}(t)
  =\frac2R\operatorname{Im}
  \int_{\R^3}
  \nabla\chi\left(\frac{x-y}{R}\right)
  \cdot\nabla U(t,x)\overline{U(t,x)}\dd x.
\end{equation}
Consequently, if $\operatorname{supp}\chi\subset B(0,2)$,
\begin{equation}\label{eq:localized-mass-growth}
  |\mathfrak m_{R,y}'(t)|
  \lesssim R^{-1}\|\nabla U(t)\|_2
  \|U(t)\|_{L^2(B(y,2R))}.
\end{equation}
\end{lemma}

\begin{proof}
Put
\[
  \chi_{R,y}(x):=\chi\left(\frac{x-y}{R}\right),
  \qquad
  V_U(t,x):=(I_\alpha*|U(t)|^p)(x)|U(t,x)|^{p-2}.
\]
Since $V_U$ is real valued and
\[
  \partial_tU=i\Delta U+iV_UU,
\]
one obtains
\begin{align}
  \frac{\dd}{\dd t}\mathfrak m_{R,y}(t)
  &=2\operatorname{Re}
  \int\chi_{R,y}\,\partial_tU\,\overline U\dd x\notag\\
  &=2\operatorname{Re}
  \int\chi_{R,y}
  \bigl(i\Delta U+iV_UU\bigr)\overline U\dd x\notag\\
  &=-2\operatorname{Im}
  \int\chi_{R,y}\,\Delta U\,\overline U\dd x,
  \label{eq:mass-derivative-cancel}
\end{align}
because
\[
  \operatorname{Re}\bigl(iV_U|U|^2\bigr)=0.
\]
Integrating the last term in \eqref{eq:mass-derivative-cancel} by parts,
\begin{align*}
  -\int\chi_{R,y}\Delta U\,\overline U\dd x
  &=\int\chi_{R,y}|\nabla U|^2\dd x
  +\int\nabla\chi_{R,y}\cdot\nabla U\,\overline U\dd x,
\end{align*}
and the first integral is real.  Since
\[
  \nabla\chi_{R,y}(x)
  =R^{-1}\nabla\chi\left(\frac{x-y}{R}\right),
\]
we obtain \eqref{eq:localized-mass-derivative}.  If $\operatorname{supp}\chi\subset B(0,2)$, then
\begin{align*}
  |\mathfrak m_{R,y}'(t)|
  &\le\frac{2\|\nabla\chi\|_\infty}{R}
  \int_{B(y,2R)}|\nabla U(t,x)|\,|U(t,x)|\dd x\notag\\
  &\le\frac{2\|\nabla\chi\|_\infty}{R}
  \|\nabla U(t)\|_2
  \|U(t)\|_{L^2(B(y,2R))}.
\end{align*}
This is \eqref{eq:localized-mass-growth}.
\end{proof}

\section{Fixed-center localized Hartree virial identity}

Except in statements formulated for a generic solution $u$, $U$ denotes a bounded-scale compact critical element. Choose a real-valued radial function $a\in C^\infty(\R^3)$ satisfying
\begin{equation}\label{eq:cutoff-a}
  a(x)=|x|^2\quad(|x|\le1),
  \qquad
  a(x)=\text{constant}\quad(|x|\ge2).
\end{equation}
For $R>0$ and $z\in\R^3$, define
\begin{equation}\label{eq:aRz}
  a_{R,z}(x):=R^2a\left(\frac{x-z}{R}\right).
\end{equation}
The definitions \eqref{eq:cutoff-a} and \eqref{eq:aRz} imply
\begin{align}
  &a_{R,z}(x)=|x-z|^2 && (|x-z|\le R),
  \label{eq:aR-quadratic}\\
  &\operatorname{supp}\nabla a_{R,z}\subset B(z,2R),
  \qquad |\nabla a_{R,z}|\lesssim R,
  \label{eq:aR-gradient}\\
  &|\partial_{jk}a_{R,z}|+|\Delta a_{R,z}|\lesssim1,
  \label{eq:aR-second}\\
  &|\Delta^2a_{R,z}|
  \lesssim R^{-2}\one_{\{R\le|x-z|\le2R\}}.
  \label{eq:aR-fourth}
\end{align}
The fixed-center virial action is
\begin{equation}\label{eq:virial-action}
  \Vcal_{R,z}(t)
  :=2\operatorname{Im}
  \int_{\R^3}\nabla a_{R,z}(x)\cdot\nabla U(t,x)\overline{U(t,x)}\dd x.
\end{equation}
The action \eqref{eq:virial-action} is finite for $U(t)\in\dot H^1$, because $U(t)$ belongs to $L^2$ on bounded sets by \eqref{eq:local-L2-energy}.

For a real $C^2$ weight $b$, define
\begin{equation}\label{eq:Qa-def}
  Q_b(x,y)
  :=\frac{(\nabla b(x)-\nabla b(y))\cdot(x-y)}{|x-y|^2},
  \qquad x\ne y.
\end{equation}
For $Q_b$ defined in \eqref{eq:Qa-def}, global Lipschitz continuity of $\nabla b$ gives
\begin{equation}\label{eq:Qa-bound}
  |Q_b(x,y)|\le\|\nabla^2b\|_{L^\infty}.
\end{equation}
For a differentiable weight $b$, write $b_j:=\partial_jb$ and $b_{jk}:=\partial_{jk}b$. Repeated spatial indices $j,k$ are summed over $\{1,2,3\}$. We also use the Kronecker symbol
\begin{equation*}
  \delta_{jk}:=\begin{cases}1,&j=k,\\0,&j\ne k.\end{cases}
\end{equation*}

We first record the exact virial identity for a general smooth real-valued weight.

\begin{lemma}[Generalized Hartree virial identity]\label{lem:Hartree-virial-identity}
Let $u$ be a smooth decaying solution of \eqref{eq:Hartree}, let $b\in C^4(\R^3)$ be real valued with bounded derivatives of the orders occurring below, and set $\rho:=|u|^p$. Define
\[
  \Vcal_b(t):=2\operatorname{Im}\int_{\R^3}
  \nabla b(x)\cdot\nabla u(t,x)\,\overline{u(t,x)}\dd x.
\]
Then
\begin{align}
 \frac{\dd}{\dd t}\Vcal_b(t)
 ={}&4\int b_{jk}
 \operatorname{Re}(\partial_ju\overline{\partial_ku})\dd x\notag\\
 &-\int\Delta^2b\,|u|^2\dd x\notag\\
 &-\left(2-\frac4p\right)
 \int\Delta b(x)(I_\alpha*\rho)(x)\rho(x)\dd x\notag\\
 &-\frac{2(3-\alpha)}p
 \iint I_\alpha(x-y)Q_b(x,y)\rho(x)\rho(y)\dd x\dd y.
 \label{eq:Hartree-virial-general}
\end{align}
The identity extends to strong $\dot H^1$ solutions for the cutoff $b=a_{R,z}$ by approximation.
\end{lemma}

\begin{proof}
Write
\[
  V:=I_\alpha*\rho,
  \qquad F:=V|u|^{p-2}u,
  \qquad \partial_tu=i\Delta u+iF.
\]
Define the nonlinear contribution by
\begin{align}
  \mathfrak N_b
  &:=2\operatorname{Re}\int b_j
  \bigl(\partial_jF\,\overline u-F\partial_j\overline u\bigr)\dd x.
  \label{eq:virial-nonlinear-start}
\end{align}
Differentiation of the action, followed by integration by parts in the linear terms, then gives
\begin{align}
 \frac{\dd}{\dd t}
 \left(2\operatorname{Im}\int\nabla b\cdot\nabla u\,\overline u\dd x\right)
 ={}&4\int b_{jk}\operatorname{Re}(\partial_ju\overline{\partial_ku})\dd x
 -\int\Delta^2b\,|u|^2\dd x
 +\mathfrak N_b.
 \label{eq:virial-linear-plus-nonlinear}
\end{align}
Since
\begin{align*}
  \operatorname{Re}\!\left[
  \partial_j(|u|^{p-2}u)\overline u
  -|u|^{p-2}u\partial_j\overline u
  \right]
  =\frac{p-2}{p}\partial_j\rho,
\end{align*}
substitution into \eqref{eq:virial-nonlinear-start} yields
\begin{align}
  \mathfrak N_b
  &=2\int b_j(\partial_jV)\rho\dd x
  +\frac{2(p-2)}p\int b_jV\partial_j\rho\dd x\notag\\
  &=\frac4p\int\nabla b\cdot\nabla V\,\rho\dd x
  -\left(2-\frac4p\right)\int\Delta b\,V\rho\dd x.
  \label{eq:virial-nonlinear-parts}
\end{align}
The Riesz kernel satisfies
\[
  \nabla I_\alpha(\zeta)
  =-(3-\alpha)I_\alpha(\zeta)\frac{\zeta}{|\zeta|^2},
  \qquad \zeta\ne0.
\]
Interchanging $x$ and $y$ and averaging therefore gives
\begin{align}
  \int\nabla b(x)\cdot\nabla V(x)\rho(x)\dd x
  &=-\frac{3-\alpha}{2}
  \iint I_\alpha(x-y)Q_b(x,y)\rho(x)\rho(y)\dd x\dd y.
  \label{eq:virial-symmetrization}
\end{align}
Combining \eqref{eq:virial-linear-plus-nonlinear}, \eqref{eq:virial-nonlinear-parts}, and \eqref{eq:virial-symmetrization} proves \eqref{eq:Hartree-virial-general}.

For the extension, approximate the initial data and the solution on compact time intervals by smooth solutions. More explicitly, for $f\in\dot H^1(\R^3)$ choose a sequence $(f_n)_{n\ge1}\subset C_c^\infty(\R^3)$ such that
\[
  \|f_n-f\|_{\dot H^1}=\|\nabla(f_n-f)\|_2\longrightarrow0.
\]
The action is continuous along this sequence because, on $B(z,2R)$,
\[
  \|f_n-f\|_{L^2(B(z,2R))}
  \lesssim R\|\nabla(f_n-f)\|_2.
\]
The nonlinear terms are continuous because $f_n\to f$ in $L^6$, hence $|f_n|^p\to|f|^p$ in $L^{6/p}$, and the bilinear HLS form is continuous on $L^{6/p}\times L^{6/p}$. This proves the identity for strong solutions with $b=a_{R,z}$.
\end{proof}

For the exact quadratic weight $b(x)=|x-z|^2$,
\begin{equation}\label{eq:quadratic-data}
  b_{jk}=2\delta_{jk},
  \qquad \Delta b=6,
  \qquad \Delta^2b=0,
  \qquad Q_b(x,y)=2.
\end{equation}
Substituting \eqref{eq:quadratic-data} into \eqref{eq:Hartree-virial-general}, the coefficient of $\Dcal(u)$ is
\begin{align}
  -6\left(2-\frac4p\right)
  -\frac{4(3-\alpha)}p
  &=-12+\frac{12+4\alpha}{p}
  =-8,
  \label{eq:critical-coefficient-check}
\end{align}
where $p=3+\alpha$ was used in the last equality. Hence \eqref{eq:critical-coefficient-check} gives
\begin{equation}\label{eq:exact-quadratic-virial}
  \frac{\dd}{\dd t}
  \left(4\operatorname{Im}\int(x-z)\cdot\nabla u\,\overline u\dd x\right)
  =8\Kcal(u)
\end{equation}
whenever the untruncated expression is meaningful. The cutoff identity below retains this main term.

For $f\in\dot H^1(\R^3)$, $z\in\R^3$, and $R>0$, define the nonnegative Hartree localization-size functional
\begin{align}
  \mathfrak E_R^{\rm H}[f;z]
  :={}&\int_{|x-z|\ge R}|\nabla f(x)|^2\dd x
  +R^{-2}\int_{R\le|x-z|\le2R}|f(x)|^2\dd x\notag\\
  &+\Ttail{R}[f;z].
  \label{eq:Hartree-error-functional}
\end{align}
For a strong solution $u$, define the signed localized virial remainder by
\begin{equation}\label{eq:Err-definition}
  \Err_{R,z}^{\rm H}(t)
  :=\frac{\dd}{\dd t}\Vcal_{R,z}(t)-8\Kcal(u(t)).
\end{equation}

Specializing the general identity to the truncated quadratic weight gives the following main term--error decomposition.

\begin{lemma}[Localized virial error]\label{lem:localized-Hartree-error}
There is a constant $C_{a,\alpha}>0$ such that, for every strong solution $u$ and every fixed $R>0$, $z\in\R^3$,
\begin{equation}\label{eq:localized-virial-decomposition}
  \frac{\dd}{\dd t}\Vcal_{R,z}(t)
  =8\Kcal(u(t))+\Err_{R,z}^{\rm H}(t),
\end{equation}
where the remainder is defined in \eqref{eq:Err-definition} and satisfies
\begin{equation}\label{eq:localized-virial-error-bound}
  |\Err_{R,z}^{\rm H}(t)|
  \le C_{a,\alpha}\mathfrak E_R^{\rm H}[u(t);z].
\end{equation}
\end{lemma}

\begin{proof}
Fix a time $t$ and set $f:=u(t)$. Apply Lemma~\ref{lem:Hartree-virial-identity} with $b=a_{R,z}$ and subtract the full quadratic identity \eqref{eq:exact-quadratic-virial}. By \eqref{eq:aR-quadratic}, the Hessian difference vanishes on $B(z,R)$, and \eqref{eq:aR-second} gives
\begin{equation}\label{eq:kinetic-error-control}
  \left|
  4\int(\partial_{jk}a_{R,z}-2\delta_{jk})
  \operatorname{Re}(\partial_jf\overline{\partial_kf})\dd x
  \right|
  \lesssim\int_{|x-z|\ge R}|\nabla f|^2\dd x.
\end{equation}
The fourth-order term is controlled by \eqref{eq:aR-fourth}:
\begin{equation}\label{eq:biharmonic-error-control}
  \left|\int\Delta^2a_{R,z}|f|^2\dd x\right|
  \lesssim R^{-2}
  \int_{R\le|x-z|\le2R}|f|^2\dd x.
\end{equation}

For the first nonlocal term, $\Delta a_{R,z}=6$ on $B(z,R)$ and is uniformly bounded. Therefore its difference from the quadratic contribution is bounded by
\begin{equation}\label{eq:single-Hartree-error}
  C\int_{|x-z|\ge R}(I_\alpha*|f|^p)(x)|f(x)|^p\dd x
  \le C\Ttail{R}[f;z].
\end{equation}
For the double term, \eqref{eq:Qa-bound} gives $|Q_{a_{R,z}}|\lesssim1$, while $Q_{a_{R,z}}=2$ whenever both $x$ and $y$ belong to $B(z,R)$. Hence
\begin{align}
 &\left|\iint I_\alpha(x-y)
  \bigl(Q_{a_{R,z}}(x,y)-2\bigr)|f(x)|^p|f(y)|^p\dd x\dd y\right|\notag\\
 &\qquad\lesssim\Ttail{R}[f;z].
 \label{eq:double-Hartree-error}
\end{align}
Combining \eqref{eq:kinetic-error-control}, \eqref{eq:biharmonic-error-control}, \eqref{eq:single-Hartree-error}, and \eqref{eq:double-Hartree-error} proves \eqref{eq:localized-virial-error-bound}.
\end{proof}

Combining the fixed-center tail estimates with the localized error bound yields uniform smallness on a long time window.

\begin{lemma}[Uniform error control on a time window]\label{lem:uniform-Hartree-error}
Under the compactness and bounded-scale assumptions, for every $\eps>0$ there exists $A=A(\eps)\ge1$ such that, for every $T\ge0$, with
\begin{equation*}
  R_T=D_x(T)+A,
\end{equation*}
one has
\begin{equation}\label{eq:uniform-error-small}
  \sup_{0\le t\le T}
  \mathfrak E_{R_T}^{\rm H}[U(t);z]\le\eps.
\end{equation}
\end{lemma}

\begin{proof}
The combination of \eqref{eq:fixed-tail-small}, \eqref{eq:fixed-annulus-small}, and \eqref{eq:fixed-Hartree-tail-small} in Lemma~\ref{lem:fixed-center-tails} proves \eqref{eq:uniform-error-small}.
\end{proof}

At bad times a small error is unavailable, so we use the following radius-independent coarse bound.

\begin{lemma}[Uniform rough derivative bound]\label{lem:rough-virial-bound}
Let $U$ satisfy \eqref{eq:below-threshold}. There exists $C_{\rm bad}=C(U,\alpha,a)<\infty$ such that for every $R\ge1$, $Z\in\R^3$, and every time in the lifespan,
\begin{equation}\label{eq:rough-virial-bound}
  \left|\frac{\dd}{\dd t}\Vcal_{R,Z}(t)\right|\le C_{\rm bad}.
\end{equation}
In particular, $\Vcal_{R,Z}'(t)\ge-C_{\rm bad}$.
\end{lemma}

\begin{proof}
Use Lemma~\ref{lem:Hartree-virial-identity}, the uniform bounds \eqref{eq:aR-second}--\eqref{eq:aR-fourth}, and \eqref{eq:Qa-bound}. The kinetic term is bounded by $C\|\nabla U(t)\|_2^2$. The fourth-order term satisfies
\[
  \left|\int\Delta^2a_{R,Z}|U|^2\dd x\right|
  \lesssim R^{-2}\int_{R\le|x-Z|\le2R}|U|^2\dd x
  \lesssim\|U(t)\|_6^2
\]
by Lemma~\ref{lem:annular-L2}. The two Hartree terms are bounded by $C_\alpha\Dcal(U(t))$, because $\Delta a_{R,Z}$ and $Q_{a_{R,Z}}$ are uniformly bounded. Finally, \eqref{eq:D-finite} and the threshold gradient bound give a time-independent constant.
\end{proof}

At times when the compact core lies well inside the fixed ball, coercivity dominates all localization errors.

\begin{lemma}[Good-time positivity]\label{lem:good-time-positivity}
Let $U\not\equiv0$ satisfy the bounded-scale compact-element assumptions and set
\begin{equation*}
  c_{\rm good}:=4\kappa_0>0.
\end{equation*}
There exists $L_0\ge1$ such that, whenever $t\ge0$, $Z\in\R^3$, $A\ge0$, and $R>0$ satisfy
\begin{equation}\label{eq:good-time-geometry}
  |x(t)-Z|\le A,
  \qquad
  R\ge4(L_0+A+1),
\end{equation}
one has
\begin{equation}\label{eq:good-time-derivative}
  \Vcal_{R,Z}'(t)\ge c_{\rm good}.
\end{equation}
\end{lemma}

\begin{proof}
Define
\begin{equation*}
  H_6(U):=\sup_{t\ge0}\|U(t)\|_6<\infty.
\end{equation*}
Since $\omega_{\rm c}(L)\to0$ by \eqref{eq:omega-def}, choose $L_0\ge1$ so large that
\begin{equation}\label{eq:L0-good-choice}
  C_{a,\alpha}\left[
  \omega_{\rm c}(L_0)
  +C\omega_{\rm c}(L_0)^{1/3}
  +C_\alpha H_6(U)^p\omega_{\rm c}(L_0)^{p/6}
  \right]
  \le4\kappa_0.
\end{equation}
Under \eqref{eq:good-time-geometry}, $|y-Z|\ge R$ implies
\begin{equation*}
  |y-x(t)|\ge R-A\ge L_0.
\end{equation*}
The kinetic term in \eqref{eq:Hartree-error-functional} is therefore at most $\omega_{\rm c}(L_0)$; Lemma~\ref{lem:annular-L2} bounds the annular term by $C\omega_{\rm c}(L_0)^{1/3}$; and Lemma~\ref{lem:Hartree-tail-HLS} bounds the Hartree term by $C_\alpha H_6(U)^p\omega_{\rm c}(L_0)^{p/6}$. Thus \eqref{eq:L0-good-choice} and \eqref{eq:localized-virial-error-bound} give
\begin{equation*}
  |\Err_{R,Z}^{\rm H}(t)|\le4\kappa_0.
\end{equation*}
Since $8\Kcal(U(t))\ge8\kappa_0$ by \eqref{eq:K-positive-gap}, the decomposition \eqref{eq:localized-virial-decomposition} yields \eqref{eq:good-time-derivative}.
\end{proof}

The improvement of the virial endpoint estimate rests on the following normalized local-mass decay.

\begin{lemma}[Compact density modulus]\label{lem:density-modulus}
Let $U$ satisfy \eqref{eq:AP-intro} and \eqref{eq:bounded-scale-intro}. Then the function $\omega_U$ defined in \eqref{eq:density-modulus-intro} satisfies
\begin{equation}\label{eq:density-modulus-decay}
  \omega_U(R)\longrightarrow0
  \qquad(R\to\infty).
\end{equation}
\end{lemma}

\begin{proof}
Fix $0<\varepsilon<1$.  By Lemma~\ref{lem:physical-tails}, choose $L=L(\varepsilon)$ such that
\begin{equation*}
  \sup_{t\ge0}
  \int_{|y-x(t)|\ge L}|U(t,y)|^6\dd y
  \le\varepsilon^3.
\end{equation*}
Define
\begin{equation*}
  C_{\rm dens}(U):=|B(0,1)|^{2/3}
  \sup_{t\ge0}\|U(t)\|_6^2<\infty.
\end{equation*}
For $R\ge2L$, decompose
\begin{align}
  \int_{|y-x(t)|\le R}|U(t,y)|^2\dd y
  &=\int_{|y-x(t)|\le L}|U(t,y)|^2\dd y\notag\\
  &\quad+
  \int_{L<|y-x(t)|\le R}|U(t,y)|^2\dd y.
  \label{eq:density-split}
\end{align}
The two terms satisfy
\begin{align}
  \int_{|y-x(t)|\le L}|U(t,y)|^2\dd y
  &\le |B(0,L)|^{2/3}\|U(t)\|_6^2
  \le C_{\rm dens}(U)L^2,
  \label{eq:density-inner}\\
  \int_{L<|y-x(t)|\le R}|U(t,y)|^2\dd y
  &\le |B(0,R)|^{2/3}
  \left(\int_{|y-x(t)|\ge L}|U(t,y)|^6\dd y\right)^{1/3}\notag\\
  &\le C R^2\varepsilon.
  \label{eq:density-outer}
\end{align}
Divide \eqref{eq:density-split} by $R^2$, use \eqref{eq:density-inner}--\eqref{eq:density-outer}, and take the supremum in $t$:
\begin{equation}\label{eq:density-modulus-bound}
  0\le\omega_U(R)
  \le C_{\rm dens}(U)\frac{L(\varepsilon)^2}{R^2}+C\varepsilon.
\end{equation}
Passing first to $R\to\infty$ and then to $\varepsilon\downarrow0$ in \eqref{eq:density-modulus-bound} proves \eqref{eq:density-modulus-decay}.
\end{proof}

The density modulus immediately yields the endpoint estimate needed at good times.

\begin{lemma}[Density endpoint]\label{lem:density-endpoint}
Let $t\ge0$, $Z\in\R^3$, $A\ge0$, and $R>0$. If $|x(t)-Z|\le A$ and $A\le R$, then
\begin{equation}\label{eq:density-endpoint}
  |\Vcal_{R,Z}(t)|
  \lesssim_U R^2\omega_U(3R)^{1/2}.
\end{equation}
If, in addition, $U(t)\in L^2$, then
\begin{equation}\label{eq:finite-mass-endpoint}
  |\Vcal_{R,Z}(t)|
  \lesssim_U R\|U(t)\|_2.
\end{equation}
\end{lemma}

\begin{proof}
By \eqref{eq:aR-gradient},
\[
 |\Vcal_{R,Z}(t)|
 \lesssim R\|\nabla U(t)\|_2
 \left(\int_{|y-Z|\le2R}|U(t,y)|^2\dd y\right)^{1/2}.
\]
The assumption $A\le R$ implies $B(Z,2R)\subset B(x(t),3R)$, which gives \eqref{eq:density-endpoint}. Dropping the spatial restriction in the last factor gives \eqref{eq:finite-mass-endpoint}.
\end{proof}

The preceding coercivity and tail estimates combine into a positive derivative statement on intervals controlled by maximal drift.

\begin{proposition}[Positive localized virial derivative]\label{prop:positive-virial-derivative}
Let $U$ satisfy \eqref{eq:below-threshold}, \eqref{eq:AP-intro}, and \eqref{eq:bounded-scale-intro}. If $U\not\equiv0$, then there exists $A_0\ge1$ such that for every $T\ge0$, with
\begin{equation}\label{eq:RT-A0}
  R_T:=D_x(T)+A_0,
\end{equation}
one has
\begin{equation}\label{eq:positive-localized-derivative}
  \frac{\dd}{\dd t}\Vcal_{R_T,z}(t)\ge4\kappa_0,
  \qquad 0\le t\le T,
\end{equation}
where $\kappa_0$ is defined in \eqref{eq:kappa-def}.
\end{proposition}

\begin{proof}
By Lemma~\ref{lem:positive-virial-gap},
\begin{equation*}
  8\Kcal(U(t))\ge8\kappa_0.
\end{equation*}
Choose $A_0$ in Lemma~\ref{lem:uniform-Hartree-error} so that
\begin{equation*}
  C_{a,\alpha}\mathfrak E_{R_T}^{\rm H}[U(t);z]\le4\kappa_0
\end{equation*}
uniformly for $0\le t\le T$. Then \eqref{eq:localized-virial-decomposition} gives \eqref{eq:positive-localized-derivative}.
\end{proof}

We next record the complementary upper bound for the localized virial action.

\begin{proposition}[Virial action upper bound]\label{prop:virial-action-bound}
Assume
\begin{equation*}
  \sup_{t\ge0}\|U(t)\|_{\dot H^1}<\infty
\end{equation*}
and $M_q(U)<\infty$ for some $q\in[2,6]$. Then, for every $R>0$ and $t\ge0$,
\begin{equation}\label{eq:action-upper-general}
  |\Vcal_{R,z}(t)|
  \lesssim_{U,q}R^{1+\theta(q)}.
\end{equation}
In particular, for $R=R_T$ from \eqref{eq:RT-A0},
\begin{equation}\label{eq:action-upper-RT}
  \sup_{0\le t\le T}|\Vcal_{R_T,z}(t)|
  \lesssim_{U,q}R_T^{1+\theta(q)}.
\end{equation}
\end{proposition}

\begin{proof}
By \eqref{eq:aR-gradient}, Cauchy--Schwarz, and the support of $\nabla a_{R,z}$,
\begin{align}
  |\Vcal_{R,z}(t)|
  &\lesssim R\|\nabla U(t)\|_2
  \|U(t)\|_{L^2(B(z,2R))}.
  \label{eq:action-CS}
\end{align}
Applying \eqref{eq:local-L2-Lq} to the last factor gives
\begin{equation*}
  \|U(t)\|_{L^2(B(z,2R))}
  \lesssim R^{\theta(q)}M_q(U).
\end{equation*}
Substitution into \eqref{eq:action-CS} proves \eqref{eq:action-upper-general} and \eqref{eq:action-upper-RT}.
\end{proof}

\section{Optimal fixed-center window exclusion}

Throughout this section, $U$ denotes a bounded-scale compact critical element. This section establishes the fixed-center exclusion mechanism used in Theorem~\ref{thm:window-scattering}. The virial center is frozen on each long interval, so no derivative of the modulation center $x(t)$ is required.

Define
\begin{equation}\label{eq:delta-star-definition}
  \delta_*:=\min\left\{\frac14,
  \frac{c_{\rm good}}{3c_{\rm good}+C_{\rm bad}}\right\},
\end{equation}
where $c_{\rm good}$ and $C_{\rm bad}$ are given by Lemmas~\ref{lem:good-time-positivity} and~\ref{lem:rough-virial-bound}.

We combine good-time positivity, the bad-time coarse estimate, and the density endpoint bound.

\begin{proposition}[Window-modulus contradiction]\label{prop:window-modulus}
Let $U$ be a bounded-scale compact critical element and let $0<\delta_{\rm w}<\delta_*$. Suppose there are $T_n\to\infty$, $A_n\ge0$, and $Z_n\in\R^3$ such that
\begin{equation*}
  \left|\{t\in[0,T_n]:|x(t)-Z_n|>A_n\}\right|
  \le\delta_{\rm w}T_n,
\end{equation*}
and
\begin{equation}\label{eq:window-Rn}
  R_n:=4(A_n+L_0+1)+T_n^{1/4},
\end{equation}
where $L_0$ is the radius in Lemma~\ref{lem:good-time-positivity}. Assume moreover that
\begin{equation}\label{eq:window-modulus-compatibility}
  \frac{R_n^2\omega_U(3R_n)^{1/2}}{T_n}\longrightarrow0.
\end{equation}
Then $U\equiv0$.
\end{proposition}

\begin{proof}
Assume $U\not\equiv0$. By \eqref{eq:delta-star-definition}, one may choose $\sigma\in(\delta_{\rm w},1/2)$ so close to $\delta_{\rm w}$ that
\begin{equation*}
  \gamma
  :=c_{\rm good}(1-2\sigma)
  -(c_{\rm good}+C_{\rm bad})\delta_{\rm w}>0.
\end{equation*}
For each $n$, set
\begin{equation*}
  G_n:=\{t\in[0,T_n]:|x(t)-Z_n|\le A_n\},
  \qquad
  B_n:=[0,T_n]\setminus G_n.
\end{equation*}
Then
\begin{equation*}
  |B_n|\le\delta_{\rm w}T_n<\sigma T_n.
\end{equation*}
Consequently, neither $[0,\sigma T_n]$ nor $[(1-\sigma)T_n,T_n]$ can be contained in $B_n$.  We may therefore choose
\begin{equation*}
  t_{1,n}\in[0,\sigma T_n]\cap G_n,
  \qquad
  t_{2,n}\in[(1-\sigma)T_n,T_n]\cap G_n,
\end{equation*}
so that
\begin{equation*}
  t_{2,n}-t_{1,n}\ge(1-2\sigma)T_n.
\end{equation*}

The radius in \eqref{eq:window-Rn} satisfies
\[
  R_n\ge4(A_n+L_0+1),
\]
hence Lemma~\ref{lem:good-time-positivity} applies on $G_n$, while Lemma~\ref{lem:rough-virial-bound}, in the form \eqref{eq:rough-virial-bound}, applies everywhere.  Splitting the time integral gives
\begin{align}
 &\Vcal_{R_n,Z_n}(t_{2,n})
  -\Vcal_{R_n,Z_n}(t_{1,n})\notag\\
 &\quad=\int_{[t_{1,n},t_{2,n}]\cap G_n}
  \Vcal_{R_n,Z_n}'(t)\dd t
  +\int_{[t_{1,n},t_{2,n}]\cap B_n}
  \Vcal_{R_n,Z_n}'(t)\dd t\notag\\
 &\quad\ge
  c_{\rm good}
  |[t_{1,n},t_{2,n}]\cap G_n|
  -C_{\rm bad}
  |[t_{1,n},t_{2,n}]\cap B_n|\notag\\
 &\quad=
  c_{\rm good}(t_{2,n}-t_{1,n})
  -(c_{\rm good}+C_{\rm bad})
  |[t_{1,n},t_{2,n}]\cap B_n|\notag\\
 &\quad\ge
  c_{\rm good}(1-2\sigma)T_n
  -(c_{\rm good}+C_{\rm bad})\delta_{\rm w}T_n
  =\gamma T_n.
  \label{eq:window-linear-growth}
\end{align}
Both endpoints belong to $G_n$, and $A_n\le R_n$.  Lemma~\ref{lem:density-endpoint} therefore yields
\begin{align}
 &|\Vcal_{R_n,Z_n}(t_{1,n})|
  +|\Vcal_{R_n,Z_n}(t_{2,n})|\notag\\
 &\qquad\lesssim_U R_n^2\omega_U(3R_n)^{1/2}
  =o(T_n),
  \label{eq:window-endpoint-small}
\end{align}
by \eqref{eq:window-modulus-compatibility}.  Combining \eqref{eq:window-linear-growth} and \eqref{eq:window-endpoint-small},
\[
  \gamma T_n
  \le o(T_n),
\]
which is impossible for large $n$.
\end{proof}

The optimal occupation radius therefore gives the principal fixed-center exclusion statement.

\begin{proposition}[Fixed-center occupation-window exclusion]\label{prop:window-rigidity}
Let $U$ be a bounded-scale compact critical element and let $\delta_*$ be defined by \eqref{eq:delta-star-definition}. If, for some $\delta_0\in(0,\delta_*)$,
\begin{equation}\label{eq:window-main-condition}
  \liminf_{T\to\infty}
  \frac{\Aopt_T(\delta_0)}{T^{1/2}}<\infty,
\end{equation}
then $U\equiv0$. Equivalently, every nonzero bounded-scale compact critical element satisfies
\begin{equation}\label{eq:window-residual-law}
  \frac{\Aopt_T(\delta)}{T^{1/2}}\longrightarrow\infty
  \qquad\text{for every }\delta\in(0,\delta_*).
\end{equation}
\end{proposition}

\begin{proof}[Proof of Proposition~\ref{prop:window-rigidity}]
Assume \eqref{eq:window-main-condition} and $U\not\equiv0$. Choose $\delta_{\rm w}$ with
\[
  \delta_0<\delta_{\rm w}<\delta_*.
\]
There are $S_n\to\infty$ and $C_0<\infty$ such that
\begin{equation}\label{eq:Aopt-bounded-sequence}
  \Aopt_{S_n}(\delta_0)\le C_0S_n^{1/2}.
\end{equation}
Passing to a subsequence, write $T_j=S_{n_j}$. By \eqref{eq:mopt-def} and \eqref{eq:Aopt-def}, choose $A_j\le\Aopt_{T_j}(\delta_0)+1$ and $Z_j\in\R^3$ such that, for all large $j$,
\[
  \left|\{t\in[0,T_j]:|x(t)-Z_j|>A_j\}\right|
  \le\delta_{\rm w}T_j.
\]
Define the corresponding frozen radii by
\begin{equation*}
  R_j:=4(A_j+L_0+1)+T_j^{1/4}.
\end{equation*}
Equation \eqref{eq:Aopt-bounded-sequence} and this definition imply
\begin{equation}\label{eq:Rj-diffusive-bounded}
  \sup_j\frac{R_j^2}{T_j}<\infty,
  \qquad R_j\to\infty.
\end{equation}
By Lemma~\ref{lem:density-modulus},
\[
  \omega_U(3R_j)\to0.
\]
Combining this with \eqref{eq:Rj-diffusive-bounded} proves \eqref{eq:window-modulus-compatibility}. Proposition~\ref{prop:window-modulus} gives the contradiction. The necessary growth law \eqref{eq:window-residual-law} is the contrapositive.
\end{proof}

The maximal-drift version follows by centering the virial ball at the initial concentration center.

\begin{proposition}[Quantitative fixed-center drift bound]\label{prop:quantitative-drift}
Let $U\not\equiv0$ satisfy \eqref{eq:below-threshold}, \eqref{eq:AP-intro}, and \eqref{eq:bounded-scale-intro}. If $M_q(U)<\infty$ for some $q\in[2,6]$, then
\begin{equation}\label{eq:quantitative-drift-main}
  T\le C(U,q,\alpha)
  \bigl(D_x(T)+1\bigr)^{1+\theta(q)},
  \qquad T\ge1.
\end{equation}
\end{proposition}

\begin{proof}
Choose $A_0$ as in Proposition~\ref{prop:positive-virial-derivative} and, for each $T\ge1$, freeze
\begin{equation*}
  R_T:=D_x(T)+A_0
\end{equation*}
on the whole interval $[0,T]$.  Integrating \eqref{eq:positive-localized-derivative},
\begin{align}
  4\kappa_0T
  &\le\int_0^T\Vcal_{R_T,z}'(t)\dd t\notag\\
  &=\Vcal_{R_T,z}(T)-\Vcal_{R_T,z}(0)\notag\\
  &\le|\Vcal_{R_T,z}(T)|+|\Vcal_{R_T,z}(0)|.
  \label{eq:quantitative-virial-chain}
\end{align}
Proposition~\ref{prop:virial-action-bound} gives
\begin{align}
  |\Vcal_{R_T,z}(T)|+|\Vcal_{R_T,z}(0)|
  &\le C(U,q,\alpha)R_T^{1+\theta(q)}\notag\\
  &\le C(U,q,\alpha)
  \bigl(D_x(T)+1\bigr)^{1+\theta(q)},
  \label{eq:quantitative-endpoint-chain}
\end{align}
where the last constant absorbs $A_0$.  Combining \eqref{eq:quantitative-virial-chain} and \eqref{eq:quantitative-endpoint-chain} proves \eqref{eq:quantitative-drift-main}.
\end{proof}

Substituting the three most useful spatial $L^q$ controls gives the following concrete drift rates.

\begin{corollary}[Three drift layers]\label{cor:three-drift-layers}
Under the hypotheses of Proposition~\ref{prop:quantitative-drift},
\begin{equation}\label{eq:three-drift-layers}
  \begin{aligned}
  q=6&:\quad T\lesssim(D_x(T)+1)^2,\\
  q=4&:\quad T\lesssim(D_x(T)+1)^{7/4},\\
  q=2&:\quad T\lesssim D_x(T)+1.
  \end{aligned}
\end{equation}
\end{corollary}

\begin{proof}
For $q=6$, Sobolev embedding and the threshold gradient bound give $M_6(U)<\infty$, and $1+\theta(6)=2$. Direct calculation gives $1+\theta(4)=7/4$ and $1+\theta(2)=1$. Proposition~\ref{prop:quantitative-drift} then yields \eqref{eq:three-drift-layers}.
\end{proof}

\paragraph{Pointwise and averaged consequences.}
Proposition~\ref{prop:window-rigidity} applies, in particular, if along a sequence $T_n\to\infty$ there are fixed centers $Z_n$ such that
\[
  \sup_{0\le t\le T_n}|x(t)-Z_n|=O(T_n^{1/2}).
\]
It also applies to $L^r$-in-time center bounds through Markov's inequality: if, for some $r>0$ and $Z_T\in\R^3$,
\[
  \int_0^T|x(t)-Z_T|^r\dd t=O(T^{1+r/2})
\]
along a sequence, then $\Aopt_T(\delta)=O(T^{1/2})$ for every fixed $\delta>0$.

\section{Low-frequency tail exclusion}

This section establishes the low-frequency exclusion mechanism used in Theorem~\ref{thm:lowfreq-scattering}. Let $\operatorname{Id}$ denote the identity operator and define
\[
  P_{>\nu}:=\operatorname{Id}-P_{\le\nu},
  \qquad
  P_{\nu<|\nabla|\le M}:=P_{\le M}-P_{\le\nu}
  \quad(0<\nu<M).
\]
The low-frequency hypothesis first supplies finite mass and then upgrades translated $\dot H^1$ compactness to $L^2$ compactness. Throughout this section, $U$ denotes a bounded-scale compact critical element; the low-frequency hypothesis will be imposed explicitly whenever it is used.

The full low-frequency exclusion is recorded in the following proposition.
\begin{proposition}[Low-frequency exclusion]\label{prop:lowfreq-rigidity}
Let $U$ be a bounded-scale compact critical element satisfying \eqref{eq:lowfreq-main}. Then
\begin{equation}\label{eq:lowfreq-conclusions-intro}
  U\in L_t^\infty H_x^1,
  \qquad
  \mathcal K_x\Subset L^2(\R^3),
  \qquad
  \mathcal K_x\Subset\dot H^1(\R^3).
\end{equation}
If $U\not\equiv0$, then $\Mcal(U)>0$, and we may define the zero-momentum Galilean parameter by
\begin{equation}\label{eq:zero-momentum-xi-intro}
  \xi_*:=-\frac{\Pcal(U)}{\Mcal(U)}.
\end{equation}
The transformed solution remains below threshold, has zero momentum, and its transformed center
\[
  x_*(t):=x(t)+2\xi_*t
\]
satisfies
\begin{equation}\label{eq:sublinear-center-intro}
  \sup_{0\le t\le T}|x_*(t)-x_*(0)|=o(T).
\end{equation}
These conclusions contradict the finite-mass virial estimate; consequently, $U\equiv0$.
\end{proposition}

We begin the proof of the low-frequency exclusion with the elementary high--low frequency decomposition.

\begin{lemma}[Finite mass from the low-frequency condition]\label{lem:lowfreq-finite-mass}
Under \eqref{eq:lowfreq-main},
\begin{equation}\label{eq:finite-mass-conclusion}
  \sup_{t\ge0}\|U(t)\|_2<\infty.
\end{equation}
Hence $U(t)\in H^1(\R^3)$, and $\Mcal(U)$ is finite and conserved.
\end{lemma}

\begin{proof}
By \eqref{eq:lowfreq-main}, choose $\nu_0\in(0,1)$ such that
\begin{equation*}
  M_{\rm lf}:=\sup_{t\ge0}\|P_{\le\nu_0}U(t)\|_2<\infty.
\end{equation*}
Plancherel and the inequality $|\xi|^{-2}\le\nu_0^{-2}$ on $\{|\xi|>\nu_0\}$ give
\begin{align*}
  \|P_{>\nu_0}U(t)\|_2^2
  &=\int_{|\xi|>\nu_0}|\widehat U(t,\xi)|^2\dd\xi\notag\\
  &\le\nu_0^{-2}
  \int_{|\xi|>\nu_0}|\xi|^2|\widehat U(t,\xi)|^2\dd\xi\notag\\
  &\le\nu_0^{-2}\|\nabla U(t)\|_2^2.
\end{align*}
Hence, using orthogonality of the frequency pieces and the threshold gradient bound,
\begin{align*}
  \sup_{t\ge0}\|U(t)\|_2^2
  &=\sup_{t\ge0}
  \left(\|P_{\le\nu_0}U(t)\|_2^2
  +\|P_{>\nu_0}U(t)\|_2^2\right)\notag\\
  &\le M_{\rm lf}^2
  +\nu_0^{-2}\sup_{t\ge0}\|\nabla U(t)\|_2^2
  <\infty.
\end{align*}
Thus \eqref{eq:finite-mass-conclusion} holds and $U\in L_t^\infty H_x^1$. The standard persistence property of the $H^1$ local theory then gives
\begin{equation*}
  U\in C([0,\infty);H^1(\R^3)).
\end{equation*}

To justify mass conservation, let $\chi\in C_c^\infty(\R^3)$ satisfy $0\le\chi\le1$, $\chi=1$ on $B(0,1)$, and, for $R>0$, define $\chi_R(x):=\chi(x/R)$. For $0\le t_1<t_2<\infty$, Lemma~\ref{lem:local-mass-variation} gives
\begin{align}
 &\left|\int\chi_R|U(t_2)|^2\dd x
  -\int\chi_R|U(t_1)|^2\dd x\right|\notag\\
 &\quad\le\frac{C}{R}
  \int_{t_1}^{t_2}\|\nabla U(t)\|_2\|U(t)\|_2\dd t\notag\\
 &\quad\le\frac{C|t_2-t_1|}{R}
  \|U\|_{L_t^\infty\dot H_x^1}
  \|U\|_{L_t^\infty L_x^2}.
  \label{eq:mass-cutoff-limit}
\end{align}
Letting $R\to\infty$ in \eqref{eq:mass-cutoff-limit} and using dominated convergence yields
\[
  \Mcal(U(t_2))=\Mcal(U(t_1)).
\]
\end{proof}

Finite mass is supplemented by a compactness upgrade that uses the low-, middle-, and high-frequency bands.

\begin{lemma}[$L^2$ precompactness]\label{lem:lowfreq-L2-precompact}
If \eqref{eq:lowfreq-main} holds, then the translated orbit $\mathcal K_x$ is precompact in $L^2(\R^3)$.
\end{lemma}

\begin{proof}
Translations change the Fourier transform only by a unimodular factor.  Fix $\varepsilon>0$.  By \eqref{eq:lowfreq-main}, choose $0<\nu<1$ so that
\begin{equation}\label{eq:L2-precompact-low}
  \sup_{t\ge0}
  \|P_{\le\nu}U(t,\cdot+x(t))\|_2
  \le\varepsilon.
\end{equation}
Next choose $M>\nu$ so large that
\begin{equation*}
  M^{-1}\sup_{t\ge0}\|\nabla U(t)\|_2
  \le\varepsilon.
\end{equation*}
For $f,g\in\mathcal K_x$, Plancherel gives
\begin{align*}
  \|P_{>M}(f-g)\|_2
  &\le M^{-1}\|\nabla(f-g)\|_2\notag\\
  &\le2M^{-1}\sup_{t\ge0}\|\nabla U(t)\|_2
  \le2\varepsilon,
  \\
  \|P_{\nu<|\nabla|\le M}(f-g)\|_2^2
  &=\int_{\nu<|\xi|\le M}|\widehat{f-g}(\xi)|^2\dd\xi\notag\\
  &\le\nu^{-2}
  \int_{\nu<|\xi|\le M}|\xi|^2|\widehat{f-g}(\xi)|^2\dd\xi\notag\\
  &\le\nu^{-2}\|\nabla(f-g)\|_2^2.
  \notag
\end{align*}
Together with \eqref{eq:L2-precompact-low},
\begin{align}
  \|f-g\|_2
  &\le\|P_{\le\nu}(f-g)\|_2
  +\|P_{\nu<|\nabla|\le M}(f-g)\|_2
  +\|P_{>M}(f-g)\|_2\notag\\
  &\le4\varepsilon+\nu^{-1}\|\nabla(f-g)\|_2.
  \label{eq:L2-by-Hdot}
\end{align}
By Lemma~\ref{lem:translated-orbit-compact}, there exist an integer $J=J(\nu,\varepsilon)\ge1$ and functions $f_1,\dots,f_J\in\mathcal K_x$ such that every $f\in\mathcal K_x$ satisfies
\[
  \|\nabla(f-f_j)\|_2\le\nu\varepsilon
\]
for some $j\in\{1,\dots,J\}$.  Substitution into \eqref{eq:L2-by-Hdot} gives
\[
  \|f-f_j\|_2\le5\varepsilon.
\]
Thus $\mathcal K_x$ is totally bounded in $L^2$.  Its $L^2$ closure is compact.
\end{proof}

We use the following standard consequence of $L^2$ precompactness.

\begin{lemma}[Uniform tails of a precompact set]\label{lem:L2-compact-tail}
If $\mathcal C\Subset L^2(\R^3)$, then
\begin{equation}\label{eq:L2-tail-precompact}
  \lim_{L\to\infty}\sup_{f\in\mathcal C}
  \int_{|y|\ge L}|f(y)|^2\dd y=0.
\end{equation}
The analogous statement holds for the gradients of a set precompact in $\dot H^1$.
\end{lemma}

\begin{proof}
Fix $\varepsilon>0$. Cover $\mathcal C$ by finitely many $L^2$ balls of radius $\varepsilon$. Choose $L$ so that the $L^2$ tail of every center is at most $\varepsilon^2$. The triangle inequality gives a uniform tail at most $4\varepsilon^2$. Apply the same argument to $\{\nabla f:f\in\mathcal C\}\Subset L^2(\R^3;\C^3)$.
\end{proof}

Combining the $L^2$, $\dot H^1$, and critical $L^6$ tails gives the enhanced tightness needed for center control.

\begin{proposition}[Enhanced compactness]\label{prop:enhanced-compactness}
Under \eqref{eq:lowfreq-main}, for every $\eta>0$ there is $L_\eta<\infty$ such that
\begin{equation}\label{eq:enhanced-tightness}
  \sup_{t\ge0}
  \int_{|y-x(t)|\ge L_\eta}
  \bigl(|U(t,y)|^2+|\nabla U(t,y)|^2+|U(t,y)|^6\bigr)\dd y
  \le\eta.
\end{equation}
\end{proposition}

\begin{proof}
Use Lemma~\ref{lem:lowfreq-L2-precompact} and Lemma~\ref{lem:L2-compact-tail}, specifically \eqref{eq:L2-tail-precompact}, for the mass tail, Lemma~\ref{lem:translated-orbit-compact} for the gradient tail, and Lemma~\ref{lem:physical-tails} for the $L^6$ tail. Take the maximum of the three radii.
\end{proof}

The low-frequency hypothesis has the following convenient sufficient condition in negative regularity.

\begin{proposition}[Negative regularity criterion]\label{prop:negative-regularity-entrance}
If \eqref{eq:negative-regularity-entrance} holds for some $s>0$, then \eqref{eq:lowfreq-main} holds.
\end{proposition}

\begin{proof}
For $0<\nu\le1$,
\[
  \|P_{\le\nu}U(t)\|_2
  \le\nu^s\||\nabla|^{-s}U(t)\|_2.
\]
Take the supremum in $t$ and let $\nu\downarrow0$.
\end{proof}
Thus Proposition~\ref{prop:negative-regularity-entrance} shows that \eqref{eq:negative-regularity-entrance} is a sufficient condition for the low-frequency hypothesis \eqref{eq:lowfreq-main}.

\subsection{Galilean normalization and sublinear center drift}

We first record momentum conservation in the regularity class used below. For a smooth, rapidly decaying solution, set $\rho:=|U|^p$. For each $j\in\{1,2,3\}$, let $\partial_{x_j}$ and $\partial_{y_j}$ denote differentiation with respect to the indicated coordinate of the two integration variables. Differentiation and integration by parts give
\begin{align*}
  \frac{\dd}{\dd t}\Pcal_j(U(t))
  &=-\frac{2}{p}\int_{\R^3}(I_\alpha*\rho)(x)\,\partial_j\rho(x)\dd x,
  \\
  2\int_{\R^3}(I_\alpha*\rho)\,\partial_j\rho\dd x
  &=-\iint_{\R^3\times\R^3}
  \partial_{x_j}I_\alpha(x-y)\rho(x)\rho(y)\dd x\dd y\\
  &\quad-\iint_{\R^3\times\R^3}
  \partial_{y_j}I_\alpha(x-y)\rho(x)\rho(y)\dd x\dd y
  =0.
\end{align*}
Hence $\Pcal(U(t))=\Pcal(U(0))$. Approximation by smooth $H^1$ solutions extends this identity to the strong solutions considered here.

For $\xi\in\R^3$ define
\begin{equation}\label{eq:Galilean-transform}
  U^\xi(t,y)
  :=e^{i(y\cdot\xi-t|\xi|^2)}U(t,y-2\xi t).
\end{equation}
Since the Hartree nonlinearity depends only on $|U|$, $U^\xi$ solves \eqref{eq:Hartree}. For $U\in H^1$, with mass and momentum defined in \eqref{eq:mass-momentum-intro},
\begin{align}
  \Mcal(U^\xi)&=\Mcal(U),\label{eq:Galilean-mass}\\
  \Pcal(U^\xi)&=\Pcal(U)+\xi\Mcal(U),\label{eq:Galilean-momentum}\\
  \Ecal(U^\xi)&=\Ecal(U)+\xi\cdot\Pcal(U)
  +\frac12|\xi|^2\Mcal(U).\label{eq:Galilean-energy}
\end{align}
The identities \eqref{eq:Galilean-mass}, \eqref{eq:Galilean-momentum}, and \eqref{eq:Galilean-energy} imply that, for $\xi=\xi_*$ from \eqref{eq:zero-momentum-xi-intro},
\begin{equation}\label{eq:Galilean-zero}
  \Pcal(U^{\xi_*})=0,
  \qquad
  \Ecal(U^{\xi_*})=\Ecal(U)-\frac{|\Pcal(U)|^2}{2\Mcal(U)},
\end{equation}
and
\begin{equation}\label{eq:Galilean-gradient}
  \|\nabla U^{\xi_*}(t)\|_2^2
  =\|\nabla U(t)\|_2^2-\frac{|\Pcal(U)|^2}{\Mcal(U)}.
\end{equation}
Thus the threshold inequalities are preserved. The transformed center is
\begin{equation}\label{eq:transformed-center}
  x_*(t):=x(t)+2\xi_*t.
\end{equation}
Enhanced $H^1$ compactness is preserved because modulation by the fixed frequency $\xi_*$ is continuous in $H^1$.

The frequency-space input and the resulting physical-space geometry are summarized in the following figure.
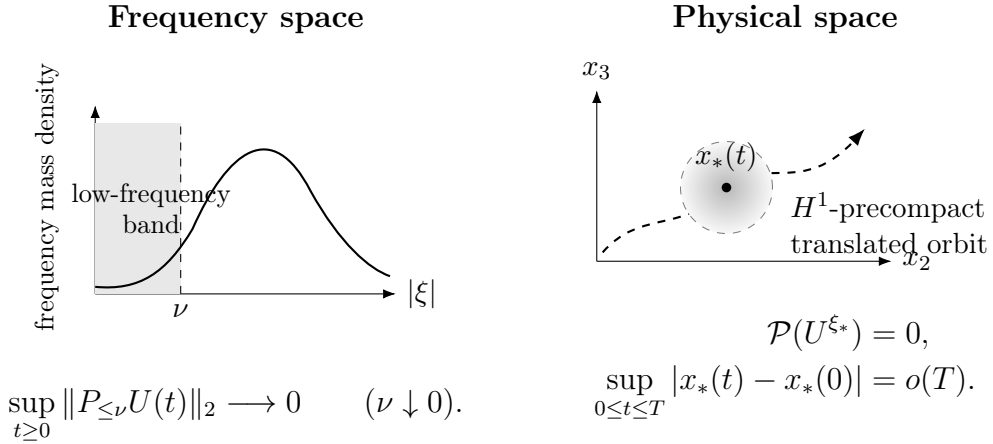
\begin{figure}[H]
\centering
\begin{minipage}[t]{0.47\textwidth}
\centering
\textbf{Frequency space}\par\medskip
\begin{tikzpicture}[x=0.78cm,y=0.78cm,>=Latex,every node/.style={font=\small}]
  \draw[->] (0,0) -- (5.1,0) node[right] {$|\xi|$};
  \draw[->] (0,0) -- (0,3.2);
  \node[rotate=90,font=\footnotesize] at (-0.82,1.65) {frequency mass density};
  \fill[gray!18] (0,0) rectangle (1.45,2.9);
  \draw[dashed] (1.45,0) -- (1.45,2.9);
  \node[below] at (1.45,0) {$\nu$};
  \draw[thick]
    (0,0.12) .. controls (0.75,0.05) and (1.20,0.35) .. (1.65,1.10)
    .. controls (2.25,2.45) and (3.00,3.00) .. (3.65,1.75)
    .. controls (4.15,0.85) and (4.65,0.42) .. (5.00,0.30);
  \node[align=center,font=\footnotesize] at (0.95,1.45) {low-frequency\\band};
\end{tikzpicture}
\[
  \sup_{t\ge0}\|P_{\le\nu}U(t)\|_2\longrightarrow0
  \qquad(\nu\downarrow0).
\]
\end{minipage}
\hfill
\begin{minipage}[t]{0.47\textwidth}
\centering
\textbf{Physical space}\par\medskip
\begin{tikzpicture}[x=0.78cm,y=0.78cm,>=Latex,every node/.style={font=\small}]
  \coordinate (O) at (0.35,0.30);
  \draw[->] (O) -- (5.35,0.30) node[right] {$x_2$};
  \draw[->] (O) -- (0.35,3.20) node[above] {$x_3$};
  \draw[dashed,thick,->]
    (0.45,0.45) .. controls (1.10,1.20) and (1.85,0.75) .. (2.55,1.55)
    .. controls (3.15,2.10) and (3.85,1.40) .. (4.90,2.55);
  \coordinate (Xt) at (2.55,1.55);
  \shade[inner color=gray!55,outer color=gray!5] (Xt) circle (0.78);
  \draw[gray,dashed] (Xt) circle (0.78);
  \fill (Xt) circle (1.8pt);
  \node[above=2pt of Xt] {$x_*(t)$};
  \node[align=left,anchor=west] at (3.45,0.95) {$H^1$-precompact\\translated orbit};
\end{tikzpicture}
\[
  \begin{aligned}
  \Pcal(U^{\xi_*})&=0,\\
  \sup_{0\le t\le T}|x_*(t)-x_*(0)|&=o(T).
  \end{aligned}
\]
\end{minipage}
\caption{Low-frequency rigidity geometry. The condition $\sup_{t\ge0}\|P_{\le\nu}U(t)\|_2\to0$ as $\nu\downarrow0$ controls the infrared mass. Together with bounded-scale compactness it yields $H^1$ precompactness. The Galilean transform $U^{\xi_*}$ has zero momentum, and truncated barycenter estimates imply sublinear motion of the transformed center. The finite-mass endpoint bound is then incompatible with the linear-in-time growth generated by the positive localized virial derivative.}
\label{fig:low-frequency-rigidity}
\end{figure}
Figure~\ref{fig:low-frequency-rigidity} records the complementary Fourier- and physical-space components of the argument.

The local mass conservation law is
\begin{equation}\label{eq:local-mass-conservation}
  \partial_t|U|^2+2\nabla\cdot\operatorname{Im}(\overline U\nabla U)=0.
\end{equation}
Choose an odd $\Phi\in C^\infty(\R)$ satisfying
\begin{equation}\label{eq:Phi-properties}
  \Phi(s)=s\ (|s|\le1),
  \quad |\Phi(s)|\le2,
  \quad 0\le\Phi'(s)\le1,
  \quad \Phi'(s)=0\ (|s|\ge2).
\end{equation}
For $R>0$ and $j\in\{1,2,3\}$, set
\begin{equation}\label{eq:localized-center-mass}
  B_{R,j}(t):=\int_{\R^3}R\Phi(y_j/R)|U(t,y)|^2\dd y.
\end{equation}
If $\Pcal(U)=0$, then applying \eqref{eq:local-mass-conservation} to \eqref{eq:localized-center-mass} and using \eqref{eq:Phi-properties} gives
\begin{align}
  B_{R,j}'(t)
  &=2\int_{\R^3}\Phi'(y_j/R)
  \operatorname{Im}\bigl(\overline{U(t,y)}\,\partial_jU(t,y)\bigr)\dd y\notag\\
  &=2\int_{\R^3}\bigl(\Phi'(y_j/R)-1\bigr)
  \operatorname{Im}\bigl(\overline{U(t,y)}\,\partial_jU(t,y)\bigr)\dd y,
  \label{eq:Bprime-identity}
\end{align}
where the second equality uses $\Pcal_j(U)=0$. By \eqref{eq:Bprime-identity}, since $\Phi'(y_j/R)-1=0$ for $|y_j|\le R$ and $|\Phi'-1|\le1$,
\begin{equation}\label{eq:Bprime-tail}
  |B_{R,j}'(t)|
  \le2
  \left(\int_{|y_j|\ge R}|U(t,y)|^2\dd y\right)^{1/2}
  \left(\int_{|y_j|\ge R}|\nabla U(t,y)|^2\dd y\right)^{1/2}.
\end{equation}

The truncated barycenter estimate now converts zero momentum and enhanced tightness into sublinear center motion.

\begin{lemma}[Sublinear drift at zero momentum]\label{lem:sublinear-drift}
Let $U$ satisfy
\begin{equation*}
  U\in C([0,\infty);H^1(\R^3))\cap L_t^\infty H_x^1,
  \qquad \Pcal(U)=0,
\end{equation*}
and assume \eqref{eq:enhanced-tightness}. After a fixed spatial translation, assume $x(0)=0$. Then
\begin{equation}\label{eq:sublinear-drift}
  D_x(T)=o(T).
\end{equation}
\end{lemma}

\begin{proof}
Fix $j\in\{1,2,3\}$ and $T\ge1$, and set
\begin{equation*}
  m:=\Mcal(U)>0.
\end{equation*}
We first verify that $x$ is bounded on $[0,T]$. Otherwise, there are $t_n\in[0,T]$ and $t_*\in[0,T]$ such that, after passing to a subsequence,
\begin{equation*}
  t_n\to t_*,
  \qquad |x(t_n)|\to\infty.
\end{equation*}
Let $L_{m/4}$ be the radius furnished by \eqref{eq:enhanced-tightness} with $\eta=m/4$. Since $U\in C([0,T];H^1)$,
\begin{align*}
  \frac{3m}{4}
  &\le\int_{B(x(t_n),L_{m/4})}|U(t_n,y)|^2\dd y\notag\\
  &\le2\|U(t_n)-U(t_*)\|_2^2
  +2\int_{B(x(t_n),L_{m/4})}|U(t_*,y)|^2\dd y
  \longrightarrow0,
\end{align*}
which is impossible. Therefore the quantity
\begin{equation*}
  D_j(T):=\sup_{0\le t\le T}|x_j(t)|
\end{equation*}
is finite. Choose $t_{j,T}\in[0,T]$ so that
\begin{equation}\label{eq:Dj-almost-max}
  |x_j(t_{j,T})|\ge D_j(T)-1.
\end{equation}
Fix $\eta$ with
\begin{equation}\label{eq:eta-small-for-drift}
  0<\eta<\frac{m}{36},
\end{equation}
choose $L_\eta$ from \eqref{eq:enhanced-tightness}, and define
\begin{equation}\label{eq:R-barycenter-choice}
  R:=4\bigl(D_j(T)+L_\eta+1\bigr).
\end{equation}
For every $t\in[0,T]$ and $|y_j|\ge R$, one has
\begin{align}
  |y-x(t)|
  &\ge |y_j-x_j(t)|
\notag\\
  &\ge |y_j|-|x_j(t)|
\notag\\
  &\ge R-D_j(T)
\notag\\
  &\ge L_\eta.
  \label{eq:coordinate-tail-inclusion}
\end{align}
Consequently, \eqref{eq:Bprime-tail}, \eqref{eq:enhanced-tightness}, and \eqref{eq:coordinate-tail-inclusion} imply
\begin{align}
  |B_{R,j}'(t)|
  &\le2
  \left(\int_{|y-x(t)|\ge L_\eta}|U(t,y)|^2\dd y\right)^{1/2}
  \left(\int_{|y-x(t)|\ge L_\eta}|\nabla U(t,y)|^2\dd y\right)^{1/2}\notag\\
  &\le2\eta.
  \label{eq:Bprime-eta}
\end{align}
Integrating \eqref{eq:Bprime-eta} gives
\begin{equation}\label{eq:B-difference-small}
  |B_{R,j}(t_{j,T})-B_{R,j}(0)|\le2\eta T.
\end{equation}

On the compact core $|y-x(t)|\le L_\eta$, \eqref{eq:R-barycenter-choice} gives
\begin{equation*}
  |y_j|
  \le |x_j(t)|+L_\eta
  \le D_j(T)+L_\eta
  <\frac R2,
\end{equation*}
so $R\Phi(y_j/R)=y_j$. Splitting the integral in \eqref{eq:localized-center-mass} into the core and its complement, and using $|\Phi|\le2$, yields
\begin{align}
 &|B_{R,j}(t)-m x_j(t)|
\notag\\
 &\quad\le
  \int_{|y-x(t)|\le L_\eta}|y_j-x_j(t)|\,|U(t,y)|^2\dd y
\notag\\
 &\qquad+
  \int_{|y-x(t)|>L_\eta}
  \bigl(|R\Phi(y_j/R)|+|x_j(t)|\bigr)|U(t,y)|^2\dd y
\notag\\
 &\quad\le L_\eta m+\bigl(2R+D_j(T)\bigr)\eta.
  \label{eq:B-center-comparison}
\end{align}
Apply \eqref{eq:B-center-comparison} at $t=t_{j,T}$ and $t=0$, use $x_j(0)=0$, \eqref{eq:Dj-almost-max}, and \eqref{eq:B-difference-small}. Since \eqref{eq:R-barycenter-choice} gives
\begin{equation*}
  4R+2D_j(T)
  =18D_j(T)+16L_\eta+16,
\end{equation*}
we obtain
\begin{align}
  m\bigl(D_j(T)-1\bigr)
  &\le2\eta T+2L_\eta m
  +\bigl(4R+2D_j(T)\bigr)\eta
\notag\\
  &=2\eta T+2L_\eta m
  +\bigl(18D_j(T)+16L_\eta+16\bigr)\eta.
  \label{eq:Dj-preabsorption}
\end{align}
Define
\begin{equation*}
  C_{m,\eta}:=m+2L_\eta m+(16L_\eta+16)\eta.
\end{equation*}
Rearranging \eqref{eq:Dj-preabsorption} and using \eqref{eq:eta-small-for-drift},
\begin{align}
  (m-18\eta)D_j(T)
  &\le2\eta T+C_{m,\eta},
  \label{eq:Dj-absorbed}\\
  m-18\eta&\ge\frac m2.
  \notag
\end{align}
Divide \eqref{eq:Dj-absorbed} by $T$ and let $T\to\infty$:
\begin{equation*}
  \limsup_{T\to\infty}\frac{D_j(T)}T
  \le\frac{4\eta}{m}.
\end{equation*}
Finally let $\eta\downarrow0$ to obtain $D_j(T)=o(T)$. Since
\begin{equation*}
  D_x(T)
  \le D_1(T)+D_2(T)+D_3(T),
\end{equation*}
we obtain \eqref{eq:sublinear-drift}.
\end{proof}

The finite-mass virial estimate then rules out the resulting zero-momentum compact element.

\begin{proposition}[Finite-mass zero-momentum rigidity]\label{prop:zero-momentum-rigidity}
No nonzero bounded-scale compact critical element $U\in C([0,\infty);H^1(\R^3))\cap L_t^\infty H_x^1$ can simultaneously satisfy \eqref{eq:enhanced-tightness} and $\Pcal(U)=0$.
\end{proposition}

\begin{proof}
Assume, to the contrary, that $U$ is nonzero.  Translate in space so that $x(0)=0$.  By Lemma~\ref{lem:sublinear-drift},
\begin{equation*}
  D_x(T)=o(T).
\end{equation*}
To choose the tail tolerance quantitatively, set $\eta_n:=2^{-n}$ for $n\ge1$. Select an increasing sequence $T_n\to\infty$ such that
\begin{equation*}
  T_n\ge nL_{\eta_n},
  \qquad
  T_{n+1}\ge T_n+1.
\end{equation*}
For $T\in[T_n,T_{n+1})$, define
\begin{equation*}
  \eta_T:=\eta_n.
\end{equation*}
Then $\eta_T\downarrow0$ and
\begin{equation}\label{eq:diagonal-tail-radius}
  0\le\frac{L_{\eta_T}}{T}
  \le\frac{L_{\eta_n}}{T_n}
  \le\frac1n
  \longrightarrow0.
\end{equation}
For $T\ge T_1$, define
\begin{equation*}
  R_T:=4\bigl(D_x(T)+L_{\eta_T}+1\bigr)+T^{1/2}.
\end{equation*}
Combining Lemma~\ref{lem:sublinear-drift} with \eqref{eq:diagonal-tail-radius}, we obtain
\begin{equation}\label{eq:zero-momentum-R-sublinear}
  \frac{R_T}{T}
  \le4\frac{D_x(T)}T
  +4\frac{L_{\eta_T}}T
  +\frac4T+T^{-1/2}
  \longrightarrow0.
\end{equation}
Moreover, for $0\le t\le T$ and $|y|\ge R_T$,
\begin{align*}
  |y-x(t)|
  &\ge |y|-|x(t)|
  \ge R_T-D_x(T)\notag\\
  &\ge L_{\eta_T},
\end{align*}
so the kinetic, annular, and Hartree errors in Lemma~\ref{lem:localized-Hartree-error} are $o_T(1)$ uniformly on $[0,T]$.  Using \eqref{eq:K-positive-gap}, for all sufficiently large $T$,
\begin{equation*}
  \Vcal_{R_T,0}'(t)
  \ge8\kappa_0-4\kappa_0
  =4\kappa_0,
  \qquad 0\le t\le T.
\end{equation*}
Integrating and applying the finite-mass endpoint estimate,
\begin{align*}
  4\kappa_0T
  &\le\Vcal_{R_T,0}(T)-\Vcal_{R_T,0}(0)\notag\\
  &\le|\Vcal_{R_T,0}(T)|+|\Vcal_{R_T,0}(0)|\notag\\
  &\lesssim_U R_T\|U\|_{L_t^\infty L_x^2}
  =o(T),
\end{align*}
where the last equality follows from \eqref{eq:zero-momentum-R-sublinear}.  This contradiction proves the proposition.
\end{proof}

\begin{proof}[Proof of Proposition~\ref{prop:lowfreq-rigidity}]
Lemmas~\ref{lem:lowfreq-finite-mass} and~\ref{lem:lowfreq-L2-precompact} prove \eqref{eq:lowfreq-conclusions-intro}; Proposition~\ref{prop:enhanced-compactness} gives \eqref{eq:enhanced-tightness}. Apply the Galilean transformation \eqref{eq:Galilean-transform} with $\xi=\xi_*$. The transformed center is given by \eqref{eq:transformed-center}. Equations \eqref{eq:Galilean-zero}--\eqref{eq:Galilean-gradient} preserve the threshold branch, and the transformed translated orbit remains precompact in $H^1$. Its momentum is zero. Lemma~\ref{lem:sublinear-drift} gives \eqref{eq:sublinear-center-intro}, and Proposition~\ref{prop:zero-momentum-rigidity} excludes the transformed solution. Since the Galilean transformation is invertible, the original nonzero compact element cannot exist.
\end{proof}

\section{Proofs of the scattering criteria}

\begin{proof}[Proof of Theorem~\ref{thm:window-scattering}]
Assume that the asserted scattering conclusion fails for some $u_0\in\mathfrak A$. By the bounded-scale compactness reduction, after reversing time when necessary, there exists a nonzero forward-global bounded-scale compact critical element $U$ satisfying one of the alternatives in Theorem~\ref{thm:window-scattering}.

Suppose first that \eqref{eq:window-scattering-condition} holds. Since $\delta_n\downarrow0$, choose $n_0$ so that $\delta_{n_0}<\delta_*$, where $\delta_*$ is defined by \eqref{eq:delta-star-definition}. For every $n\ge n_0$, monotonicity of $\Aopt_T(\delta)$ in $\delta$ gives
\begin{align*}
  \Aopt_{T_n}(\delta_{n_0})
  &\le \Aopt_{T_n}(\delta_n),\\
  \liminf_{T\to\infty}
  \frac{\Aopt_T(\delta_{n_0})}{T^{1/2}}
  &\le \liminf_{n\to\infty}
  \frac{\Aopt_{T_n}(\delta_n)}{T_n^{1/2}}<\infty.
\end{align*}
Proposition~\ref{prop:window-rigidity} therefore gives $U\equiv0$, a contradiction.

Suppose next that \eqref{eq:drift-scattering-condition} holds. Proposition~\ref{prop:quantitative-drift} implies
\[
  1\le C(U,q,\alpha)
  \frac{\bigl(D_x(T)+1\bigr)^{1+\theta(q)}}{T},
  \qquad T\ge1,
\]
which contradicts \eqref{eq:drift-scattering-condition}. The exponent identities in Corollary~\ref{cor:three-drift-layers} give the sufficient rates in \eqref{eq:drift-rate-examples}.

Thus no solution with initial datum in $\mathfrak A$ can fail the conclusion \eqref{eq:global-and-finite-Z-intro}. Proposition~\ref{prop:finite-Z-scattering}, applied forward and backward, yields $u_\pm\in\dot H^1$ and \eqref{eq:two-sided-scattering-intro}; the limiting form of \eqref{eq:Duhamel-intro} gives \eqref{eq:duhamel-scattering-intro}.
\end{proof}

\begin{proof}[Proof of Theorem~\ref{thm:lowfreq-scattering}]
Assume again that the scattering conclusion fails. The bounded-scale compactness reduction produces a nonzero forward-global bounded-scale compact critical element $U$. Under \eqref{eq:lowfreq-main}, Proposition~\ref{prop:lowfreq-rigidity} gives $U\equiv0$, a contradiction. If instead the stronger condition \eqref{eq:negative-regularity-entrance} is assumed, Proposition~\ref{prop:negative-regularity-entrance} first yields \eqref{eq:lowfreq-main}, and the same contradiction follows.

Consequently \eqref{eq:global-and-finite-Z-intro} holds for every $u_0\in\mathfrak A$. Proposition~\ref{prop:finite-Z-scattering} gives the states $u_\pm$ and \eqref{eq:two-sided-scattering-intro}; passing to the limit in \eqref{eq:Duhamel-intro} yields \eqref{eq:duhamel-scattering-intro}.
\end{proof}

\section*{Funding}
P.-H. Chung was supported by the Project ``Research on Nonlinear Partial Differential Equations'' (No.~2024KYCXTD018), the Special Projects in Key Areas of Guangdong Province (No.~ZDZX1088), and the Fund of Guangzhou Municipal Science and Technology (No.~202102080428).

\section*{Declaration of competing interest}
The authors declare that they have no known competing financial interests or personal relationships that could have appeared to influence the work reported in this paper.

\section*{Data availability}
No data were used for the research described in this article.

\end{document}